\documentclass[a4paper]{amsart}
\usepackage{stix}
\usepackage{hyperref}
\usepackage{amsxtra}
\usepackage[czech,english]{babel}
\usepackage{mdwtab}
\usepackage{mathtools}
\usepackage[OT1]{fontenc}
\usepackage[all]{xy}
\xyoption{barr}
\usepackage{comment}
\usepackage{ifthen}
\usepackage[normalem]{ulem}
\newcommand\stout{\bgroup\markoverwith{\color{red}{\rule[0.5ex]{2pt}{1.4pt}}}\ULon}
\excludecomment{future}

\newcommand{\seq}[1]{\pmb{\mathsf{#1}}}
\newcommand{\Xseq}[1]{\mathsf{#1}}

\newcommand{\lift}[2][]{{\rm L}_{#1}(#2)}
\newcommand{\ball}[2][]{{\rm N}_{#2}\ifthenelse{\equal{A#1}{A}}{}{^{#1}}}

\newcommand{\bbbn}{\mathbb{N}}

\usepackage{color}
\definecolor{shadecolor}{gray}{.85}%
\definecolor{tintedcolor}{gray}{.80}%
\usepackage{framed}%
\definecolor{mytintedcolor}{gray}{.95}%
\makeatletter
\newdimen\svparindent
\newcounter{tmpthm}
\newenvironment{mytinted}{%
  \MakeFramed {\FrameRestore}}%
{\endMakeFramed}
\newenvironment{great}{%
\fboxsep=12pt\relax

        \vbox\bgroup\begin{mytinted}%
        \list{}{\leftmargin=12pt\rightmargin=2\leftmargin\leftmargin=\z@\topsep=\z@\relax}%
        \expandafter\item\parindent=\svparindent
        \relax}%
{\endlist\end{mytinted}\egroup}
\makeatother
\newlength{\graphshift}

\newtheorem{theorem}{Theorem}[subsection]
\newtheorem{proposition}[theorem]{Proposition}

\newtheorem{lemma}[theorem]{Lemma}

\theoremstyle{definition}
\newtheorem{definition}[theorem]{Definition}
\newtheorem{example}[theorem]{Example}

\theoremstyle{remark}

\newtheorem{conjecture}[theorem]{Conjecture}

\newsavebox{\mybox}
\newlength{\mydepth}
\newlength{\myheight}

{\medskip\begin{lrbox}{\mybox}\hspace{1em}\begin{minipage}{\textwidth-1em}}%
{\end{minipage}\end{lrbox}%
  \settodepth{\mydepth}{\usebox{\mybox}}%
  \settoheight{\myheight}{\usebox{\mybox}}%
  \addtolength{\myheight}{\mydepth}%
  \noindent\makebox[0pt]{\rule[-\mydepth]{1pt}{\myheight}}%
  \usebox{\mybox}\medskip}
\begin{document}

\title{Local-Global Convergence,\\
an analytic and structural approach}
\author{Jaroslav Ne{\v s}et{\v r}il}
\address{Jaroslav Ne{\v s}et{\v r}il\\
Computer Science Institute of Charles University (IUUK and ITI)\\
   Malostransk\' e n\' am.25, 11800 Praha 1, Czech Republic}
\email{nesetril@iuuk.mff.cuni.cz}
\thanks{Supported by grant ERCCZ LL-1201 
and CE-ITI P202/12/G061, and by the European Associated Laboratory ``Structures in
Combinatorics'' (LEA STRUCO)}

\author{Patrice Ossona~de~Mendez}
\address{Patrice~Ossona~de~Mendez\\
Centre d'Analyse et de Math\'ematiques Sociales (CNRS, UMR 8557)\\
  190-198 avenue de France, 75013 Paris, France
	--- and ---
Computer Science Institute of Charles University (IUUK)\\
   Malostransk\' e n\' am.25, 11800 Praha 1, Czech Republic
   --- and ---
 Department of Mathematics, Zhejiang Normal University, China  
   }
\email{pom@ehess.fr}
\thanks{Supported by grant ERCCZ LL-1201 and by the European Associated Laboratory ``Structures in
Combinatorics'' (LEA STRUCO), and partially supported by ANR project Stint under reference ANR-13-BS02-0007}
\dedicatory{In memory Bohuslav Balcar.}
\date{\today}
.
\begin{abstract}
Based on methods of structural convergence  we provide a unifying view of local-global convergence,  fitting to model theory and analysis. The general approach outlined here provides a possibility to extend the theory of local-global convergence to graphs with unbounded degrees. As an application, we extend previous results on continuous clustering of local convergent sequences and prove the existence of modeling quasi-limits for local-global convergent sequences of nowhere dense graphs.
\end{abstract}
\maketitle
\begin{quote}
{\em The true logic of the world is in the calculus of probabilities}\\
\flushright James Clerk Maxwell
\end{quote}

\section{Introduction}
The study of graph limits recently gained a strong interest, motivated both by the study of large networks and the emerging studies of real evolving networks. The different notions of graph limit ({\em left convergence} \cite{Borgs2012a, Borgs2005,  Borgs20081801, Borgs2012, Lov'asz2006}, {\em local convergence} \cite{Benjamini2001}), and the basic notions of graph similarity on which they are based, opens a vast panorama.
These frameworks have in common to be built on statistics of locally defined motives when the vertices of the graphs in the sequence are sampled uniformly and independently. 
A unified framework for the study of convergence of structures has been introduced by the authors in \cite{CMUC}.

In this setting the  notion of convergence is, in essence, model theoretic, and relies on the the following notions:

For a $\sigma$-structure $\mathbf A$ and a first-order formula $\phi$ (in the language of $\sigma$, with free variables $x_1,\dots,x_p$),  we denote by $\phi(\mathbf A)$ the {\em satisfying set} of $\phi$ in $\mathbf A$:
$$
\phi(\mathbf A)=\{(v_1,\dots,v_p)\in A^p:\ \mathbf{A}\models\phi(v_1,\dots,v_p)\},
$$
and we define the  {\em Stone pairing} of $\phi$ and $\mathbf A$ as 
the probability 
\[
\langle\phi,\mathbf{A}\rangle=\frac{|\phi(\mathbf{A})|}{|A|^p}
\]
that $\mathbf{A}$ satisfies $\phi$ for a (uniform independent) random interpretation of the random variables.

A sequence $\seq{A}=(\mathbf{A}_n)_{n\in\bbbn}$ of finite $\sigma$-structures
 is {\em ${\rm FO}$-convergent} if
 the sequence \linebreak $\langle\phi,\seq{A}\rangle=(\langle\phi,\mathbf{A}_n\rangle)_{n\in\bbbn}$ converges for every first-order formula $\phi$.
One similarly defines a weakened
notion of {\em $X$-convergence} (for a fragment $X$ of first-order logic) by restricting the range of the test formulas $\phi$
to $X$. In particular, for a sequence of graphs with growing orders, the  {\em ${\rm QF}$-convergence} (that is of convergence driven by the fragment ${\rm QF}$ of quantifier-free formulas) is equivalent to left convergence, and the {\em ${\rm FO}^{\rm local}$-convergence} (that is of convergence driven by the fragment ${\rm FO}^{\rm local}$ of the so-called local formulas) is equivalent to local convergence when restricting to sequences of graphs with bounded degrees \cite{CMUC}. The study of $X$-convergence of structures and the related problems is called shortly ``structural limits''. It extends all previously considered types of non-geometric convergence of combinatorial structures (two above and also, e.g. permutations) to general  relational structures. A survey of structural limits can be found in \cite{modeling,Nevsetvril2014}.

In order to strengthen the notion of local convergence of sequences of connected bounded degree graphs, the notion of {\em local-global convergence} has been introduced by Hatami and Lov\'asz \cite{hatami2014limits}, based on the framework introduced in \cite{bollobas2011sparse}. The orignal definition of local-global convergence, which is based on the total variation distance between the distributions of colored neighborhoods of radius $r$ in  the graphs in the sequence, is admittedly quite technical (see Definitions~\ref{def:lHc} and~\ref{def:LG-HL}). Roughly speaking, the strengthening of local convergence into local-global convergence relates to the extension of the considered properties from first-order logic to existential monadic second order logic.

In this paper we introduce the notion of {\em lift-Hausdorff convergence}, which is defined generally for graphs and relational structures, based on a simple subsequence completion property (see Definition~\ref{def:LG}). The basic notions underlying this definition are the model theoretic notions of {\em lift} of a structure (that one can see as an augmentation of a structure by colors, new relations, etc.) and the dual notion of {\em shadow} consisting in forgetting all the additional relations of a lift (see Fig.~\ref{fig:lift}). Note that the notions of lift and shadow are close to model theoretic notions of expansion and reduct.

\begin{figure}[h]
\begin{center}
	\includegraphics[width=.5\textwidth]{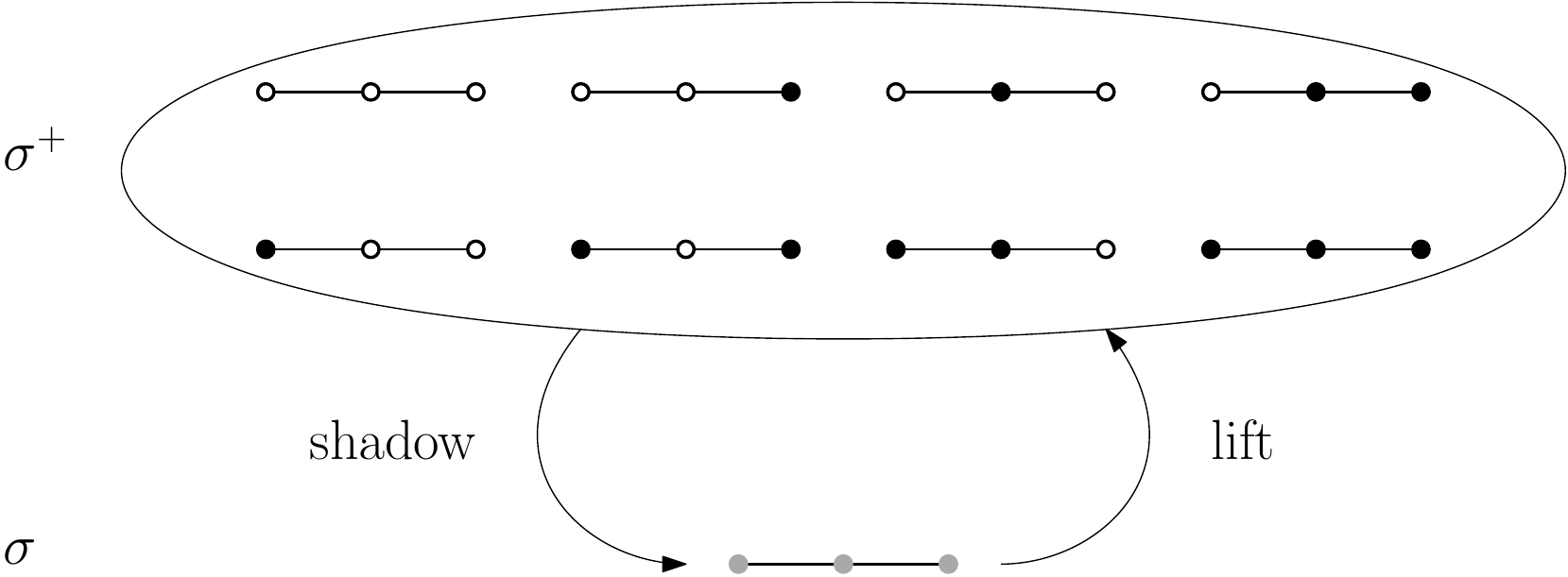}
\end{center}
\caption{The {\em shadow} operation consists in forgetting some of the relations, while the {\em lift} operation consists in adding some relations in all possible ways.}
\label{fig:lift}
\end{figure}

 A lift/shadow pair is determined by the data of two signatures $\sigma\subset\sigma^+$, the signature of the shadows and the signature of the lifts. Given a lift/shadow pair, the corresponding notion of lift-Hausdorff convergence is (roughly) defined as follows:
		A sequence $\seq{A}=(\mathbf A_n)_{n\in\mathbb N}$ is {\em lift-Hausdorff} convergent if, for every convergent subsequence $\seq{A}_f^+=(\mathbf A_{f(n)}^+)_{n\in\mathbb N}$ of lifts  there exists a (full) convergent sequence  $\seq{A^*}$ of lifts extending $\seq{A}_f^+$ .
(In this definition, the condition that $\seq{A}_f$ is a subsequence of lifts means that $\mathbf A_{f(n)}^+$ is a lift of $\mathbf A_{f(n)}$ and $f:\mathbb N\rightarrow\mathbb N$ is an increasing function defining a subsequence of indices; the condition that $\seq{A}^*$ extends $\seq{A}_f^+$ means that  $\mathbf A_{f(n)}^*=\mathbf A_{f(n)}^+$ for every integer $n$.)

We prove that this notion indeed extends the usual notion of local-global convergence in the case of monadic lifts and gives a characterization of lift-Hausdorff convergent sequences as Cauchy sequences for an appropriate metric (Theorem~\ref{thm:Cauchy}).
Note that the case of monadic lifts can be interpreted by means of first-order formulas with set parameters. 

We give a basic representation theorem for corresponding limits, which is based on the representation theorem for structural limits \cite{CMUC}: every limit of a lift-Hausdorff convergent sequence maybe represented by non-empty closed subset of the space of probability measures on some Stone space (Theorem~\ref{thm:LGrep}). In Section~\ref{sec:quasi} we discuss the possible notion of a limit object for lift-Hausdorff convergent sequences. 
Two possible notions can be considered:
\begin{itemize}
	\item a strong notion of limit $\mathbf L$ for a sequence $\seq{A}$, such that for every 
	convergent sequence $\seq{A}^+$ of lifts there exists an admissible lift $\mathbf L^+$ of $\mathbf L$ that is the limit of $\seq{A}^+$ and, conversely, for every admissible lift $\mathbf L^+$ of $\mathbf L$ there exists a convergent sequence $\seq{A}^+$ of lifts with limit $\mathbf L^+$;
	\item a weaker notion of limit $\mathbf L$ (in the spirit of the representation of local-global limits by graphings \cite{hatami2014limits}) where we ask that for every $\epsilon>0$ and every convergent sequence $\seq{A}^+$ of lifts there exists an admissible lift $\mathbf L^+$ of $\mathbf L$ such that for every formula $\phi$ we have 
	\begin{equation}
	\label{eq:close}
	\bigl|\langle\phi,\mathbf L^+\rangle-\lim_{n\rightarrow\infty}\langle\phi,\mathbf A_n^+\rangle\bigr|<c_\phi \epsilon
	\end{equation}
	(where $c_\epsilon$ is a positive constant depending only on $\phi$) 
	and, conversely, for every admissible lift $\mathbf L^+$ of $\mathbf L$ there exists a convergent sequence $\seq{A}^+$ of lifts such that \eqref{eq:close} holds.
\end{itemize}

In these definitions, an essential difficulty lies in the definition of admissible lifts. In the case of local-global convergent sequences of bounded degree graphs with a graphing limit, the notion of admissible lift corresponds to Borel colorings. In a more general setting, an admissible lift of a modeling should (at least) be a modeling itself. However, Borel colorings of a fixed graphing does not induce a closed subset of unimodular probability measures, as noticed in 
 \cite{hatami2014limits} and, more generally, probability measures associated to modeling lifts of a fixed modeling does not form a closed subset, what is problematic for the stronger notion of limit.
 
 When considering the weaker (standard) definition of a limit, the reverse direction appears to be quite difficult to handle in general. Indeed, in the bounded degree case, the $r$-neighborhood of a set with measure at most $\epsilon$ has measure a most $d^r\epsilon$, what allows to easily approximates Borel colorings using a countable base. However such a property does not hold in the general unbounded degree case. 
 
 For this reason, we only consider one direction in the definition of limit:  a {\em quasi-limit} $\seq{A}$ is a modeling such that for every $\epsilon>0$, every formula $\phi$ and every convergent sequence $\seq{A}^+$ of lifts there exists a modeling lift $\mathbf L^+$ of $\mathbf L$ such that \eqref{eq:close} holds. 
In this setting, we prove  that modelings (introduced in \cite{limit1}) are not only limit objects for FO-convergent nowhere dense sequences (as proved in \cite{modeling_jsl}) but also quasi-limits for (monadic) lift-Hausdorff  convergent nowhere dense sequences (see Theorem~\ref{thm:quasilimit}).

\medskip

This paper is organized as follows:
\renewcommand{\contentsname}{\vspace{-2\baselineskip}}
\setcounter{tocdepth}{2}
\tableofcontents

Let us end this introduction by few remarks. The lifts involved in our local-global structural convergence are all monadic (and can be seen as coloring of vertices). It follows that the expressive power of such lifts is restricted to embeddings and classes which are hereditary. If we would consider more general lifts, such as coloring of the edges \footnote{This would require to add some first-order restrictions on the lifts, what would not fundamentally change the framework presented in this paper.} then we could represent monomorphisms (not induced substructures) which in turn leads to monotone classes. Monotone classes of graphs which have modeling limits (of which graphing limits are a particular case) were characterized in \cite{modeling_jsl} and coincides with nowhere dense classes (of graphs). This also coincides (in the case of monotone classes of graphs) with the notion of NIP and stable classes \cite{Adler2013} (See also \cite{SurveyND}). For hereditary classes the structure theory and the existence of modeling limits is more complicated (see \cite{modeling}) and local-global convergence seems to provide a useful framework.

\section{Preliminaries}	
\subsection{Relational Structures and First-Order Logic}
A {\em signature} $\sigma$ is a set of relation symbols with associated arities. In this paper we will consider countable signatures. A {\em $\sigma$-structure} $\mathbf A$ is defined by its {\em domain} $A$, which is a set, and by interpreting each relation symbol $R\in\sigma$ of arity $k$ as a subset of $A^k$. We denote by ${\rm Rel}(\sigma)$ the set of all finite $\sigma$-structures and by $\mathscr{R}\rm{el}(\sigma)$ the class of all $\sigma$-structures.

A {\em first-order formula} $\phi$ in the language of $\sigma$-structures is a formula constructed using disjunction, conjunction, negation and quantification over elements, using the relations in $\sigma$ and the equality symbol. A variable used in a formula $\phi$ is {\em free} if it is not bound by a quantifier. We always assume that free variables are named $x_1,\dots,x_n,\dots$ and we consider formulas obtained by renaming the free variables as distinct. For instance, $x_1=x_2$ and $x_2=x_3$ are distinct formulas. We also consider two constants, $0$ and $1$ to denote the {\em false} and {\em true} statements.
We denote by ${\rm FO}(\sigma)$ the (countable) set of all first-order formulas in the language of $\sigma$-structures. The conjunction and disjunction of formulas $\phi$ and $\psi$ are denoted by $\phi\wedge\psi$ and $\phi\vee\psi$, and the negation of $\phi$ is denoted by $\neg\phi$. We say that two formulas $\phi$ and $\psi$ are {\em logically equivalent}, which we denote by $\phi\iff\psi$, if one can infer one from the other (i.e. $\phi\vdash\psi$ and $\psi\vdash\phi$). Note that in first-order logic the notions of syntactic and semantic equivalence coincides. In this context we denote by $[\phi]$ the equivalence class of $\phi$ with respect to logical equivalence. It is easily checked that $\mathcal B_\sigma={\rm FO}(\sigma)/\!\!\!\!\iff$ is a countable Boolean algebra with minimum $0$ and maximum $1$, which is called the {\em Lindenbaum-Tarki algebra}  of ${\rm FO}(\sigma)$.

In this paper we consider special fragments of first-order logic (see Table~\ref{tab:frag}).

\begin{table}
\begin{tabular}{|c|l|}
	 \hlx{sshv}
Symbol&\multicolumn{1}{c|}{Fragment}\\
\hlx{vhv}
${\rm FO}(\sigma)$ or ${\rm FO}$&All first order formulas\\
${\rm FO}_p(\sigma)$ or ${\rm FO}_p$&All first order formulas with free variables within $\{x_1,\dots,x_p\}$\\
${\rm FO}_0(\sigma)$ or ${\rm FO}_0$&Sentences\\
${\rm FO}^{\rm local}(\sigma)$ or ${\rm FO}^{\rm local}$&Local formulas\\
${\rm FO}_p^{\rm local}(\sigma)$ or ${\rm FO}_p^{\rm local}$&Local formulas with free variables within $\{x_1,\dots,x_p\}$\\
${\rm QF}(\sigma)$ or ${\rm QF}$&Quantifier free formulas\\
\hlx{vhss}
\end{tabular}
\caption{Principal fragments of first-order logic considered in this paper.} \label{tab:frag}
\end{table}

The Lindenbaum-Tarski algebra of a fragment $X\subseteq {\rm FO}(\sigma)$ will be denoted by $\mathcal B_\sigma^X$. For instance, $\mathcal B_\sigma^{\rm QF}={\rm QF}(\sigma)/\!\!\!\!\iff$.
\subsection{Functional Analysis}
Basic facts from Functional Analysis, which will be used in this paper, are recalled now.

A {\em standard Borel space} is a Borel space associated to 
a Polish space, i.e. a measurable space $(X,\Sigma)$ such that  there exists a metric on $X$ making it a separable complete metric space with $\Sigma$ as its Borel $\sigma$-algebra. Typical examples of standard Borel spaces are $\mathbb R$ and the Cantor space. Note that according to Maharam's theorem, all uncountable standard Borel spaces are (Borel) isomorphic. (The authors cannot resist the temptation to mention  Balcar's award-winning work \cite{Balcar2005} in this context.)

In this paper, we shall mainly consider compact separable metric spaces. Note that if $(M,{\rm d})$ is a compact separable metric space, then linear functionals on the space of real continuous functions on $M$ can be represented, thanks to Riesz-Markov-Kakutani representation theorem, by Borel measures on $M$.
We denote by $P(M)$ the space of all probability measures on $M$. A sequence of probability measures $(\mu_n)_{n\in\mathbb N}$ is {\em weakly convergent} if
$\int f\,{\rm d}\mu_n$ converges for every (real-valued) continuous function\footnote{We do not have to assume that $f$ has compact support as we assumed that $(M,d)$ is compact.} $f$. Weak convergence defines the {\em weak topology} of $P(M)$, and (as we assumed that $(M,d)$ is a compact separable metric space) this space is compact, separable, and metrizable by the L\'evy-Prokhorov metric (based on the metric ${\rm d}$):
\[
{\rm d}^{\rm LP}(\mu_1,\mu_2)=\inf\Bigl\{\epsilon>0\mid \mu_1(A)\leq \mu_2(A^\epsilon)+\epsilon\text{ and }\mu_2(A)\leq \mu_1(A^\epsilon)+\epsilon\text{ for all Borel }A\Bigr\},
\]
where $A^\epsilon=\{x\in M\mid (\exists y\in M)\ {\rm d}(x,y)<\epsilon\}$.

The Hausdorff metric is defined on the space of nonempty closed bounded subsets of a metric space.
Consider a compact metric space  $(M,{\rm d})$, and let  $\mathscr C_M$ be the space of non-empty closed  subsets of $M$ endowed with the {\em Hausdorff metric} defined by
\[
{\rm d}^{\rm H}(X,Y)=\max\bigl(\adjustlimits\max_{x\in X}\min_{y\in Y}{\rm d}(x,y),\max_{y\in Y}\min_{x\in X}{\rm d}(x,y)\bigr).
\]

One of the most important properties of Hausdorff metric is that the space of non-empty closed subsets of a compact set is also compact (see \cite{hausdorff1962set}, and \cite{urysohn1951works} for an independent proof). Hence the space $(\mathscr C_M,{\rm d}^{\rm H})$ is compact.

We can use the inverse function of a surjective continuous function from a compact metric space $(M,{\rm d})$  to a (thus compact) Hausdorff space $T$  to isometrically embed the space $\mathscr C_T$ (of non-empty closed subsets of $T$) into $(\mathscr C_M,{\rm d}^{\rm H})$. Then, using the natural injection $\iota:T\rightarrow\mathscr C_T$ (defined by $\iota(x)=\{x\}$) we pull back the Hausdorff distance on $\mathscr C_M$ into $T$: 

\[{\rm d}_T(x,y)={\rm d}^{\rm H}(f^{-1}(x),f^{-1}(y)).\] 

The situation is summarized in the following diagram.
$$\xy\xymatrix{
(M,{\rm d})\ar@{->>}^-f[r]&T\ar@{ >->}^-{\iota}[d]\ar@{ >->}_{f^{-1}}[dl]\\
(\mathscr C_M,{\rm d}^{\rm H})&\mathscr C_T\ar@{ >->}^-{\hat f^{-1}}[l]
}\endxy
$$
In this diagram $f^{-1}$ denotes the mapping from $T$ to $\mathscr C_M$ and $\hat f^{-1}$ the corresponding mapping from $\mathscr C_T$ to $\mathscr C_M$ defined by $\hat f^{-1}(X)=\{y\mid f(y)\in X\}$.
Also remark that the metric ${\rm d}_T$ defined on $T$ is usually not compatible with the original topology of $T$.

For the topology defined by the metric ${\rm d}_T$, one can
 define the compactification of $T$, which may be identified with the closure of the image of $T$ (by $f^{-1}\circ\iota$) in $\mathscr C_M$.

We shall make use of the following folklore result, which we prove here for completeness.
\begin{lemma}
\label{lem:contpush}
	Let $X,Y$ be compact standard Borel spaces and 
	let $f:X\rightarrow Y$ be continuous. Let $P(X)$ and $P(Y)$ denote the metric space of probability measures on $X$ and $Y$ (with L\'evy-Prokhorov metric).
	
	Then the pushforward by $f$,  that is the mapping $f_*:P(X)\rightarrow P(Y)$ defined by $f_*(\mu)=\mu\circ f^{-1}$, is continuous.
\end{lemma}
\begin{proof}
	Assume $\mu_n\Rightarrow\mu$ is a weakly convergent sequence of measures in $P(X)$. Then for every continuous function $g:Y\rightarrow\mathbb R$ it holds
\begin{align*}
	\int_Y g(y)\,{\rm d}f_*(\mu)(y)&=
	\int_X g\circ f(x)\,{\rm d}\mu(x)\\
	&=\lim_{n\rightarrow\infty} \int_X g\circ f(x)\,{\rm d}\mu_n(x)\\
	&=\lim_{n\rightarrow\infty} \int_Y g(y)\,{\rm d}f_*(\mu_n)(y)
\end{align*}
Hence $f_*(\mu_n)\Rightarrow f_*(\mu)$.
\end{proof}

\subsection{Sequences}	
	In this paper we denoted sequences by sans serif letters. In particular, we denote by $\seq{A}$ a sequence of structures	$\seq{A}=(\mathbf A_n)_{n\in\mathbb N}$, and by $\Xseq{X}=(X_n)_{n\in\mathbb N}$ a sequence of sets $X_n$, where $X_n$ is a subset of the domain $A_n$ of $\mathbf A_n$.
	
	Subsequences will by denoted by $\seq{A}_f$ and $\Xseq{X}_f$, where $f$ is meant to be a strictly increasing function $f:\mathbb N\rightarrow\mathbb N$, and represent the sequences $(\mathbf A_{f(n)})_{n\in\mathbb N}$ and $(X_{f(n)})_{n\in\mathbb N}$. Note that $(\seq{A}_f)_g=\seq{A}_{g\circ f}$.
	
	In order to simplify the notations, we extend binary relations and standard constructions to sequences by applying them component-wise. For instance
	$\Xseq{X}\subseteq \Xseq{Y}$ means $(\forall n\in\mathbb N)\ X_n\subseteq Y_n$, 
	$\Xseq{X}\cap\Xseq{Y}$ represents the sequence
	$(X_n\cap Y_n)_{n\in\mathbb N}$,
	and if $f:{\rm Rel}(\sigma)\rightarrow\mathbb R$ is a mapping then
	$f(\seq{A})$ represents the sequence $(f(\mathbf A_n))_{n\in\mathbb N}$.
	
We find these notations extremely helpful for our purposes.

\subsection{Basics of Structural Convergence}
Let $\sigma$ be a countable signature, let $X$ be a fragment of ${\rm FO}(\sigma)$.
For $\phi\in X$ with free variables within $x_1,\dots,x_p$ and $\mathbf A$, we denote by $\langle\phi,\mathbf A\rangle$ the probability that $\phi$ is satisfied in $\mathbf A$ for a random assignment of elements of $A$ to the free variables of $\phi$ (for an independent and uniform random choice of the assigned elements), that is:
\[
\langle\phi,\mathbf A\rangle=\frac{\bigl|\bigr\{\overline v\in A^p\mid \mathbf A\models\phi(\overline v)\bigr\}\bigr|}{|A|^p}.
\]
(Note that the presence of unused free variables does not change the value in the next equation.)
In the special case where $\phi$ is a sentence, we get
\[
\langle\phi,\mathbf A\rangle=\begin{cases}
1&\text{if }\mathbf A\models\phi,\\
0&\text{otherwise}	
\end{cases}
\]

Two $\sigma$-structures $\mathbf A$ and $\mathbf B$ are {\em $X$-equivalent}, what we denote by $\mathbf A\equiv_X\mathbf B$, if we have $\langle\phi,\mathbf A\rangle=\langle\phi,\mathbf B\rangle$ for every $\phi\in X$.
\begin{example}
Let ${\rm QF}^-$ be the fragment of quantifier-free formulas that do not use equality. 

\flushleft If a graph $G$ is obtained from a graph $H$ by blowing each vertex into $k$ vertices (i.e. if $G$ is the lexicographic product of $H$ by an edgeless graph of order $k$) then $G\equiv_{\rm QF^-}H$.

If $X$ is a fragment including ${\rm QF}$ or ${\rm FO}_0$, then $\equiv_X$ is trivial on finite relational structures.
\end{example}

The case of ${\rm FO}^{\rm local}_1$-equivalence is settled by the next proposition.
\begin{proposition}
For any two finite $\sigma$-structures $\mathbf A$ and $\mathbf B$ we have that 
$\mathbf A\equiv_{{\rm FO}_1^{\rm local}}\mathbf B$
if and only if there exists a finite $\sigma$-structure $\mathbf C$ and two positive integers $a$ and $b$ such that 
$\mathbf A$ is isomorphic to $a$ copies of $\mathbf C$ and $\mathbf B$ is isomorphic to $b$ copies of $\mathbf C$.
\end{proposition}
\begin{proof}
Let $\mathbf F_1,\dots,\mathbf F_n,\dots$ be an enumeration of the finite $\sigma$-structures (up to isomorphism), and let
$\varphi_i(x)$ be a local formula expressing that the connected component of $x$ is isomorphic to $\mathbf F_i$ (i.e. that the ball of radius $|F_i|+1$ around $x$ is isomorphic to $F_i$). Then 
$\langle\varphi_i,\mathbf A\rangle$ is equal to the product of 
$|\mathbf F_i|/|\mathbf A|$ by the number of connected components of $\mathbf A$ isomorphic to $\mathbf F_i$. Thus there exists a positive integer $q$ and non-negative integers $p_1,\dots,p_n,\dots$ such that
$\langle\varphi_i,\mathbf A\rangle=p_i/q$ and the set  of all positive
$p_i$ values is setwise coprime. Then if $\mathbf C$ consists in thus union (over $i$) of $p_i$ copies of $\mathbf F_i$, it is immediate that $\mathbf A$ and $\mathbf B$ consists in a positive number of copies of $\mathbf C$.
\end{proof}

A sequence $\seq{A}$ of $\sigma$-structures is {\em $X$-convergent} if $\langle\phi,\seq{A}\rangle=(\langle\phi,\mathbf A_n\rangle)_{n\in\mathbb N}$ converges for each $\phi\in X$.
This provides a unifying to left and local convergence, as mentioned in the introduction: left convergence coincides with ${\rm QF}^-$-convergence and local convergence with ${\rm FO}^{\rm local}$-convergence (when restricted to graphs with bounded degrees).
The term of {\em structural convergence} covers the general notions of $X$-convergence.

The basic result of \cite{CMUC}, which is going to provide us a guideline for a proper generalization of local-global convergence is the representation theorem for structural limits in terms of probability measures. We adopt \cite{CMUC} to the setting of this paper.

\subsection{The Representation Theorem for Structural Limits}
\label{sec:rep}

For a countable signature $\sigma$ and a fragment $X$ of ${\rm FO}(\sigma)$ we denote by $S_X^\sigma$ the Stone dual of the Lindenbaum-Tarski $\mathcal B_\sigma^X$ of $X$, which is a compact Polish space. Recall that the points of $\mathcal B_\sigma^X$ are the maximal consistent subsets of $X$ (or equivalently the ultrafilters on $\mathcal B_\sigma^X$). The topology of $S_X^\sigma$ 
is generated by the base of the clopen subsets of $S_X^\sigma$, which are in bijection with the formulas in $X$ by 
$$\phi\mapsto K(\phi)=\bigl\{t\in S_X^\sigma\mid \phi\in t\bigr\}.$$

In the setting of this paper we work with metric (and, notably, pseudo-metric) spaces. First note that the 
 topology of $S_X^\sigma$ is metrizable by the several metrics, including the metrics we introduce now.

A {\em chain covering} of $X$ is an increasing sequence $\mathfrak X=(X_1,X_2,\dots)$ of finite sets (i.e. $X_1\subset X_2\subset \dots\subset X_n\subset\dots$) such that every formula in $X$ is logically equivalent to a formula in $\bigcup_{i\geq 1}X_i$.
The metric $\delta_{\mathfrak X}$ induced by $\mathfrak X$ on $S_\sigma^X$ is defined by

\begin{great}
\begin{equation}
\delta_{\mathfrak X}(t_1,t_2)=\inf\bigl\{1/n\mid (t_1\vartriangle t_2)\cap X_n=\emptyset\bigr\},
\end{equation}
\end{great}

where $t_1\vartriangle t_2$ stands for the symmetric difference of the sets $t_1$ and $t_2$.

First-order limits (shortly ${\rm FO}$-limits) and, more generally, $X$-limits can be uniquely represented by a probability measure $\mu$ on the Stone space $S$ dual to the Lindenbaum-Tarski algebra of the  formulas. This can be formulated as follows.

\begin{theorem}[\cite{CMUC}]
Let $\sigma$ be a countable signature,
	let $X$ be a fragment of ${\rm FO}(\sigma)$ closed under disjonction, conjunction and negation, let $\mathcal B^X_\sigma$ be the Lindenbaum-Tarski algebra of $X$, and let $S^X_\sigma$ be the Stone dual of $\mathcal B^X_\sigma$.
	
Then there is a map $T^X_\sigma$ from the space ${\rm Rel}(\sigma)$ of finite $\sigma$-structures to the space of $P(S_\sigma^X)$ of probability measures on the Stone space $S^X_\sigma$, such that for every $\mathbf A\in{\rm Rel}(\sigma)$ and every $\phi\in X$ we have
\begin{equation}
	\label{eq:rep}
	\langle\phi,\mathbf A\rangle=\int_{S_\sigma^X}\mathbf 1_\phi(t)\,{\rm d}\mu_{\mathbf A}(t),
\end{equation}
where $\mu_{\mathbf A}=T^X_\sigma(\mathbf A)$ and  $\mathbf 1_\phi$ is the indicator function of the clopen subset $K(\phi)$ of $S^X_\sigma$ dual to the formula $\phi\in X$ in Stone duality, i.e. 
\[
\mathbf 1_\phi(t)=\begin{cases}
	1&\text{if }\phi\in t\\
	0&\text{otherwise}.
\end{cases}
\]
Additionally, if the fragment $X$ includes ${\rm FO}_0$ or ${\rm QF}$ then the mapping $T_\sigma^X$ is one-to-one\footnote{Note that in \cite{CMUC} the condition on $X$ was erroneously omitted.}.

In this setting, a sequence $\seq{A}$ of finite
$\sigma$-structures is $X$-convergent if and only if the measures $T_\sigma^X(\mathbf{A}_n)$ converge weakly to some measure $\mu$. Then  for every first-order formula $\phi\in X$ we have
 \begin{equation}
 \label{eq:mu}
 \int_{S^X_\sigma} \mathbf 1_\phi(t)\,{\rm d}\mu(t)=
 \lim_{n\rightarrow\infty}\int_{S^X_\sigma} \mathbf 1_\phi(t)\,{\rm d}\mu_{\mathbf{A}_n}(t)=
 \lim_{n\rightarrow\infty} \langle\phi,\mathbf{A}_n\rangle,
 \end{equation}
 where $\mu_{\mathbf{A}_n}=T^X_\sigma(\mathbf A_n)$.

Assume that a subgroup $\Gamma$ of 
the group $\mathfrak S_\omega$ of permutations of $\mathbb N$  acts on the first-order formulas in $X$  by permuting the free variables. Then this action induces an action on $S_\sigma^X$, and the probability  measure $T^{X}_\sigma(\mathbf A)$ associated with a finite structure $\mathbf A$ is obviously $\Gamma$-invariant, thus so is the weak limit $\mu$ of a sequence $(T^{X}_\sigma(\mathbf A_n))_{n\in\mathbb N}$ of probability measures associated with the finite structures of an X-convergent sequence. It follows that the measure $\mu$ appearing in \eqref{eq:mu} has the property to be $\Gamma$-invariant.
\end{theorem}

 This theorem generalizes the representation of the limit of a left-convergent sequence of graphs by an infinite exchangeable random graph \cite{Aldous1981, Hoover1979} and the representation of the limit of a local-convergent sequence of bounded degree graphs by a unimodular distribution \cite{Benjamini2001}.
 Figure~\ref{fig:spacesFO} schematically depicts some of the notions related to the representation theorem.

\begin{figure}[ht]
	\begin{center}
		\includegraphics[width=\textwidth]{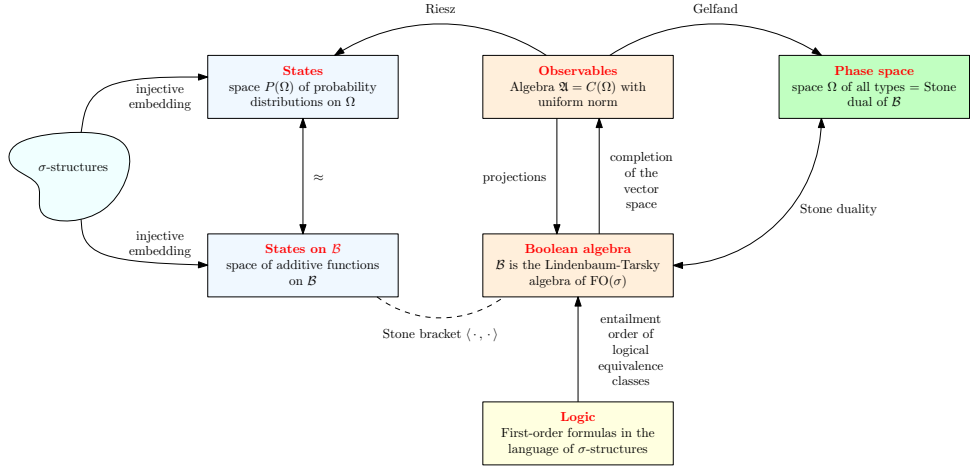}
	\end{center}
	\caption{Sketch of the spaces involved in the distributional representation of structural limits of $\sigma$-structures.}
	\label{fig:spacesFO}
\end{figure}

The weak topology of $P(S_\sigma^X)$ is metrizable by using the L\'evy-Prokhorov metrics based on the metrics $\delta_{\mathfrak X}$ (where $\mathfrak X$ is a fixed chain covering of $X$).
Using the fact that $\delta_{\mathfrak X}$ is an ultrametric, we obtain the following more practical expression for the L\'evy-Prokhorov metric ${\rm d}_{\mathfrak X}^{\rm LP}$ associated to $\delta_{\mathfrak X}$:
\begin{great}
\begin{equation}
{\rm d}_{\mathfrak X}^{\rm LP}(\mu_1,\mu_2)=
\inf_{n\in\mathbb N}\biggl\{\max\Bigl(1/n,
	\max_{\phi\in X_n}\Bigl|\int_{S_\sigma^X}\mathbf 1_\phi(t)\,{\rm d}(\mu_1-\mu_2)(t)\Bigr|\Bigr)\biggr\}.
\end{equation}
\end{great}

This metric in turn uniquely defines a pseudometric ${\rm dist}_{\mathfrak X}$ on ${\rm Rel}(\sigma)$  such  that the mapping $T_\sigma^X$ induces an isometric embedding of ${\rm Rel}(\sigma)/\!\equiv_X$ into $P(S_\sigma^X)$:

\[
	{\rm dist}_{\mathfrak X}(\mathbf A,\mathbf B)={\rm d}_{\mathfrak X}^{\rm LP}(T_\sigma^X(\mathbf A),T_\sigma^X(\mathbf B)).
\]

Note that we have the following expression for ${\rm dist}_{\mathfrak X}$:
\begin{great}
\begin{equation}
	{\rm dist}_{\mathfrak X}(\mathbf A,\mathbf B)=\inf_{n\in\mathbb N}\Bigl\{\max\bigl(1/n,\max_{\phi\in X_n}|\langle\phi,\mathbf A\rangle-\langle\phi,\mathbf B\rangle|\bigr)\Bigr\}.
\end{equation}
\end{great}

It is easily checked that, as expected, a sequence $\seq{A}$ is $X$-convergent if and only if it is Cauchy for ${\rm dist}_{\mathfrak X}$.
 
We denote by $\mathscr M^X_\sigma\subseteq P(S^X_\sigma)$ the space of probability measures on $S^X_\sigma$ associated to finite $\sigma$-structures:
\[
\mathscr M^X_\sigma=\{T^X_{\sigma}(\mathbf A)\mid \mathbf A\in{\rm Rel}(\sigma)\}.
\]
We denote by $\overline{\mathscr M^X_\sigma}$ the weak closure of $\mathscr M^X_\sigma$ in $P(S^X_\sigma)$ and   by $(\overline{{\rm Rel}(\sigma)}^X,{\rm dist}_{\mathfrak X})$ the completion of
the pseudometric space $({\rm Rel}(\sigma),{\rm dist}_{\mathfrak X})$. Note that $(\overline{{\rm Rel}(\sigma)}^X,{\rm dist}_{\mathfrak X})$ has a dense subspace naturally identified with $({\rm Rel}(\sigma)/\!\!\equiv_X,{\rm dist}_{\mathfrak X})$, and $T_\sigma^X$ induces an isometric isomorphism of $(\overline{{\rm Rel}(\sigma)}^X,{\rm dist}_{\mathfrak X})$ and $(\overline{\mathscr M^X_\sigma}, {\rm d}^{\rm LP}_{\mathfrak X})$. Consequently both spaces are separable compact metric spaces.

\section{From interpretation to lift convergence}

Our basic approach to local-global convergence is by means of lifts of structures, which demands a change of signature. By doing so we still have to preserve some functorial properties and this is done by means of interpretations.

\subsection{A Categorical Approach to Interpretations}

Interpretations of classes of relational structures in other classes of relational structures are a useful and powerful technique to transfer properties from one class of structures to another (with possibly a different signature).

First, we define interpretations syntactically (in the spirit of \cite{Lascar2009}), which allows us to organize them as a category. This functorial view will be particularly useful in our setting.

Let $\tau,\sigma$ be countable relationa	l signatures. An {\em interpretation} $\mathsf I$ of $\sigma$-structures in $\tau$-structures is a triple 
$(\nu,\eta,(\rho_R)_{R\in \sigma})$, where:
\begin{itemize}
	\item $\nu(\overline x)\in {\rm FO}(\tau)$ is a formula defined on $p$ tuples of variables $\overline x$;
	\item $\eta(\overline x,\overline y)\in {\rm FO}(\tau)$ is a formula defining an equivalence relation on $p$-tuples (satisfying $\nu$);
	\item for each relation $R\in\sigma$ of arity $k$, the formula
	$\rho_R(\overline x_1,\dots,\overline x_k)\in {\rm FO}(\tau)$ (with $|\overline x_1|=\dots=|\overline x_k|=p$) is compatible with $\eta$, meaning  $$\bigwedge_{i=1}^k\eta(\overline x_i,\overline y_i)\vdash\rho_R(\overline x_1,\dots,\overline x_k)\leftrightarrow \rho_R(\overline y_1,\dots,\overline y_k).$$
\end{itemize}

By replacing equality by $\eta$, relation $R$ by $\rho_R$ and by conditioning quantifications using $\nu$ one easily checks that 
the interpretation $\mathsf I$ allows to associate
to each formula $\phi(x_1,\dots,x_k)\in {\rm FO}(\sigma)$ a formula $\hat\phi(\overline x_1,\dots,\overline x_k)\in{\rm FO}(\tau)$. 
We define
\[L(\mathsf I):{\rm FO}(\sigma)\rightarrow{\rm FO}(\tau)\] as the mapping $\phi\mapsto\hat\phi$.

Note that we have $L(\mathsf I)(1)=\nu,L(\mathsf I)(x_1=x_2)=\eta$, and $L(\mathsf I)(R)=\rho_R$ for every $R\in\sigma$. Hence
$L(\mathsf I)$ fully determines $\mathsf I$.

This definition allows us to consider interpretations $\mathsf I:\tau\rightarrow\sigma$ as morphisms in a {\em category of interpretations}. The objects of this category are all countable relational signatures (here denoted by $\sigma,\tau$, \dots ) and morphisms $\mathsf I:\tau\rightarrow\sigma$ are triples $(\nu,\eta,(\rho_R)_{R\in \sigma})$ forming an interpretation as above. Morphisms compose as if $\mathsf I:\tau\rightarrow\sigma$ and $\mathsf J:\sigma\rightarrow\kappa$ are interpretations then we can define
\[\mathsf J\circ\mathsf I=(L(\mathsf J\circ\mathsf I)(x_1=x_1),L(\mathsf J\circ\mathsf I)(x_1=x_2),(L(\mathsf J\circ\mathsf I)(S))_{S\in\kappa}).\]

The identity (for $\sigma$) is provided by the morphism
$(x_1=x_1,x_1=x_2,(R)_{R\in\sigma})$. Thus we indeed have a category of interpretations.

A {\em basic interpretation} \cite{modeling} is an interpretation $(\nu,\eta,(\rho_R)_{R\in \sigma})$ such that $\nu(\overline x):=(\overline x=\overline x)$ and $\eta(\overline x_1,\overline x_2):=(\overline x_1=\overline x_2)$. (For instance the identity interpretation defined above is a basic interpretation.)

Note that every basic interpretation $\mathsf I:\tau\rightarrow\sigma$ induces 
a homomorphism 
\[H(\mathsf I):\mathcal B_\sigma\rightarrow\mathcal B_\tau\qquad\text{defined by}\qquad
H(\mathsf I)([\phi])=[L(\mathsf I)(\phi)],\]
 where $[\phi]$ denotes the class of $\phi$ for logical equivalence.
The mapping $H$ is actually a contravariant functor from the category of interpretations to the category of Boolean algebras.

By Stone duality theorem, the interpretation $\mathsf I$ also defines a  continuous function 
\[F(\mathsf I):S_\tau\rightarrow S_\sigma\qquad\text{defined by}\qquad
F(\mathsf I)(t)=\{\phi\mid L(\mathsf I)(\phi)\in t\}.
\]
 Note that $F$ is a covariant functor from the category of interpretations to the category of Stone spaces.

Finally, the interpretation $\mathsf I$ also defines a mapping
\[P(\mathsf I):{\mathscr R}\rm{el}(\tau)\rightarrow {\mathscr R}\rm{el}(\sigma)\] as follows:
\begin{itemize}
	\item The domain of $P(\mathsf I)(\mathbf A)$ is $\nu(\mathbf A)/\eta$, that is all the $\eta$-equivalence classes of $p$-tuples in $\nu(\mathbf A)$.
	\item For every relational symbol $R\in\sigma$ with arity $k$ (and associated formula $\rho_R$) we have
	$$P(\mathsf I)(\mathbf A)\models R([\overline v_1],\dots,[\overline v_k])\quad\text{iff}\quad \mathbf A \models\rho_R(\overline v_1,\dots,\overline v_k).$$
	(Note that this does not depend on the choice of the representatives $\overline v_1,\dots,\overline v_k$ of the $\eta$-equivalence classes $[\overline v_1],\dots,[\overline v_k]$.)
\end{itemize}
This mapping $P(\mathsf I):{\rm Rel}(\tau)\rightarrow{\rm Rel}(\sigma)$ is what is usually meant by an interpretation (of $\sigma$-structures in $\tau$-structures, see \cite{Hodges1997}).
It is easily checked that the mapping $p5\mathsf I)$ has the property that 
for every formula $\phi\in {\rm FO}(\sigma)$ with $k$ free variables and every $[\overline v_1],\dots,[\overline v_k]$ in $P(\mathsf I)(\mathbf A)$ we have
\[
P(\mathsf I)(\mathbf A)\models \phi([\overline v_1],\dots,[\overline v_k])\quad\text{iff}\quad \mathbf A\models L(\mathsf I)(\phi)(\overline v_1,\dots,\overline v_k).
\]

The interpretations, which we shall the most frequently consider, belong to the following types of basic interpretations (which are easily checked to be basic interpretations):
\begin{itemize}
	\item 
 {\em forgetful interpretations} that simply forget some of the relations, 
 \item
{\em renaming interpretations} that bijectively map a signature to another, mapping a relational symbol to a relational symbol with same arity,
\item {\em projecting interpretations} that forget some symbols and rename others.
\end{itemize}

Our categorical approach allows us to obtain a more functorial point of view:

$$\xy\xymatrix{
\tau\ar[d]_{\mathsf I}&{\rm FO}(\tau)&\mathcal B_\tau\ar@{<~>}^{\text{Stone duality}}[rr]&&S_\tau\ar[d]^{F(\mathsf I)}&{\rm Rel}(\tau)\ar@{~>}[d]_{P(\mathsf I)}\ar@{~>}[r]^{T_\tau}&\mathcal M_\tau\ar@{ (->}[r]&P(S_\tau)\ar@{~>}[d]^{F(\mathsf I)_*}\\
\sigma&{\rm FO}(\sigma)\ar[u]^{L(\mathsf I)}&\mathcal B_\sigma\ar[u]^{H(\mathsf I)}\ar@{<~>}^{\text{Stone duality}}[rr]&&S_\sigma&{\rm Rel}(\sigma)\ar@{~>}[r]^{T_\sigma}&\mathcal M_\sigma\ar@{ (->}[r]&P(S_\sigma)
}\endxy
$$
In this diagram the mapping $F(\mathsf I)_*$ is the pushforward defined by $F(\mathsf I)$ (see Lemma~\ref{lem:contpush} bellow).

One can also consider the case where we do not consider all first-order formulas. Let $X$ be a fragment of ${\rm FO}(\tau)$ and let $Y=L(\mathsf I)^{-1}(X)$.
(Note that if $X$ is closed by disjunction, conjunction and negation, so is $L(\mathsf I)^{-1}(X)$.)
 The basic interpretation $\mathsf I$ then defines a homomorphism
 \[H(\mathsf I):\mathcal B_\sigma^Y\rightarrow\mathcal B_\tau^X,\]
  which is the restriction of $H(\mathsf I)$ to $\mathcal B_\sigma^Y$. By duality, this homomorphism defines a continuous mapping 
  \[\hat F(\mathsf I):S_\tau^X\rightarrow S_\sigma^Y.\]
   In particular, if $K_\phi$ is the clopen subset of $S_\sigma^Y$ defined by $\phi\in Y$ then $\hat F(\mathsf I)^{-1}(K_\phi)$ is the clopen subset $K_{L(\mathsf I)(\phi)}$ of $S_\tau^X$. Note that we have $F(\mathsf I)=\hat F(\mathsf I)\circ\Pi_\tau^X$, where $\Pi_\tau^X$ is the natural projection from $S_\tau$ to $S_\tau^X$.

\subsection{Metric properties of interpretations}
\label{sec:met}
We have seen in the previous section that interpretations define continuous functions between Stone spaces. This property can be used to transfer convergence from one signature to another. This is done in a very general setting we introduce now.

Let $\mathsf I:\tau\rightarrow\sigma$ be a basic interpretation,
let $X$ be a fragment of ${\rm FO}(\tau)$, let $Y=L(\mathsf I)^{-1}(X)$, let $\mathfrak X$ be a chain coverings of $X$, and let $\mathfrak Y$ be a chain covering of $Y$ such that every formula in $L(\mathsf I)(Y_n)$ is logically equivalent to a formula in $X_n$.

Let us explain this choice of $Y$.

 In ${\rm Im}\,P(\mathsf I)$ we should not distinguish two finite  structures $P(\mathsf I)(\mathbf A)$ and $P(\mathsf I)(\mathbf B)$ if there exist a chain $\mathbf C_1,\dots,\mathbf C_{2n+1}$ of finite structures  such that $\mathbf C_1=\mathbf A$, $\mathbf C_{2n+1}=\mathbf B$,
$C_{2i-1}\equiv_X \mathbf C_{2i}$, and $P(\mathsf I)(\mathbf C_{2i})=P(\mathsf I)(\mathbf C_{2i+1})$ for $i=1,\dots,n$.
But $P(\mathsf I)(\mathbf A')=P(\mathsf I)(\mathbf B')$ holds if and only if $\langle\phi,\mathbf A'\rangle=\langle\phi,\mathbf B'\rangle$ for every $\phi\in L(\mathsf I)({\rm FO}(\sigma))$. Hence the conditions can be rewritten as 
\begin{align*}
\forall\phi&\in X&\langle\phi,\mathbf C_{2i-1}\rangle&=\langle\phi,\mathbf C_{2i}\rangle	\\
\forall\phi&\in L(\mathsf I)({\rm FO}(\sigma))&\langle\phi,\mathbf C_{2i}\rangle&=\langle\phi,\mathbf C_{2i+1}\rangle	
\intertext{A necessary (but maybe not sufficient) condition is obviously that}
\forall\phi&\in X\cap L(\mathsf I)({\rm FO}(\sigma))&
\langle\phi,\mathbf A\rangle&=\langle\phi,\mathbf B\rangle.
\\
\intertext{that is:}
\forall\phi&\in L(\mathsf I)^{-1}(X)&
\langle\phi,P(\mathsf I)(\mathbf A)\rangle&=\langle\phi,P(\mathsf I)(\mathbf B)\rangle,
\end{align*}
which we can rewrite as 
$P(\mathsf I)(\mathbf A)\equiv_Y P(\mathsf I)(\mathbf B)$.
This shows that the fragment $Y$ is sufficiently small to ensure the continuity of $P(\mathsf I)$. By our choice of the chain covering $\mathfrak Y$ we further get that $P(\mathsf I)$ induces a {\em short map} (that is a $1$-Lipschitz function). We summarize this in the following lemma.

\begin{lemma}
\label{lem:sm}
In the above setting and notation we have:
	\begin{great}
	\begin{equation}
		{\rm dist}_{\mathfrak Y}(P(\mathsf I)(\mathbf A),P(\mathsf I)(\mathbf B))\leq {\rm dist}_{\mathfrak X}(\mathbf A,\mathbf B).
	\end{equation}
\end{great}

This can be restated as follows:
Let $({\rm Rel}(\tau)/\!\equiv_X,{\rm dist}_{\mathfrak X})$ and $({\rm Rel}(\sigma)/\!\equiv_Y,{\rm dist}_{\mathfrak Y})$ be the quotient metric spaces induced by the pseudometric spaces $({\rm Rel}(\tau),{\rm dist}_{\mathfrak X})$ and $({\rm Rel}(\sigma),{\rm dist}_{\mathfrak Y})$.
Then the unique continuous function $\hat P(\mathsf I):({\rm Rel}(\tau)/\!\equiv_X,{\rm dist}_{\mathfrak X})\rightarrow ({\rm Rel}(\sigma)/\!\equiv_Y,{\rm dist}_{\mathfrak Y})$ such that $\hat P(\mathsf I)([\mathbf A]_X)=[P(\mathsf I)(\mathbf A)]_Y$ is a short map. 
\end{lemma}
\begin{proof}
For every pair $\mathbf A,\mathbf B$ au $\tau$-structures we have
\begin{align*}
	{\rm dist}_{\mathfrak Y}(P(\mathsf I)(\mathbf A),P(\mathsf I)(\mathbf B))&=\inf_{n\in\mathbb N}\Bigl\{\max\bigl(1/n,\max_{\phi\in Y_n}|\langle\phi,P(\mathsf I)(\mathbf A)\rangle-\langle\phi,P(\mathsf I)(\mathbf A)\rangle|\bigr)\Bigr\}\\
	&=\inf_{n\in\mathbb N}\Bigl\{\max\bigl(1/n,\max_{\psi\in L(\mathsf I)(Y_n)}|\langle\phi,\mathbf A\rangle-\langle\phi,\mathbf B\rangle|\bigr)\Bigr\}\\
	&\leq \inf_{n\in\mathbb N}\Bigl\{\max\bigl(1/n,\max_{\psi\in X_n}|\langle\phi,\mathbf A\rangle-\langle\phi,\mathbf B\rangle|\bigr)\Bigr\}\\
	&={\rm dist}_{\mathfrak X}(\mathbf A,\mathbf B).
	\end{align*}
(In particular $\mathbf A\equiv_X\mathbf B$ imply that ${\rm dist}_{\mathfrak Y}(P(\mathsf I)(\mathbf A),P(\mathsf I)(\mathbf B))=0$ thus
$P(\mathsf I)(\mathbf A)\equiv_Y P(\mathsf I)(\mathbf B)$ hence $P(\mathsf I)$ descends to the quotient and there exists a unique map 
\[\hat P(\mathsf I):({\rm Rel}(\tau)/\!\equiv_X,{\rm dist}_{\mathfrak X})\rightarrow ({\rm Rel}(\sigma)/\!\equiv_Y,{\rm dist}_{\mathfrak Y})\]
 such that $\hat P(\mathsf I)([\mathbf A]_X)=[P(\mathsf I)(\mathbf A)]_Y$.)
\end{proof}

Let $\mathscr M_{X,\mathsf I}=F(\mathsf I)_*({\mathscr M_\tau^X})$ and let $\overline{\mathscr M_{X,\mathsf I}}=F(\mathsf I)_*(\overline{\mathscr M_\tau^X})$ be the closure of $\mathscr M_{X,\mathsf I}$ in $(P(S_\sigma^Y),{\rm d}_{\mathfrak Y}^{\rm LP})$.
We tried to summarize in Fig.~\ref{fig:diag} the relations between the different (pseudo)metric spaces defined from signatures, fragments, and interpretations. 

\begin{figure}[ht]
$$
\xy\xymatrix@C=.025\textwidth{
({\rm Rel}(\tau),{\rm dist}_{\mathfrak X})\ar@{->>}[rr]^-{T_\tau^X}\ar@{->>}[dd]^{P(\mathsf I)}\ar@{->>}[dr]^{[\cdot]_X}
&&({\mathscr M^X_\tau},{\rm d}_{\mathfrak X}^{\rm LP})\ar@{ >->}[rr]\ar@{->>}'[d]_-{F(\mathsf I)_*}[dd]
&&(\overline{\mathscr M^X_\tau},{\rm d}_{\mathfrak X}^{\rm LP})\ar@{->>}_-{F(\mathsf I)_*}[dd]
\\
&({\rm Rel}(\tau)/\!\equiv_X,{\rm dist}_{\mathfrak X})\ar@{ >->>}[ur]\ar@{->>}[dd]_(0.3){\hat P(\mathsf I)}\ar@{ >->}[rr]
&&(\overline{{\rm Rel}(\tau)}^X\!\!,{\rm dist}_{\mathfrak X})\ar@{ >->>}[ur]\ar@{->>}[dd]
&&&\\
({\rm Im}\,P(\mathsf I),{\rm dist}_{\mathfrak Y})\ar@{->>}'[r][rr]^(.3){T_\sigma^Y}\ar@{ >->}[dd]\ar@{->>}[dr]^{[\cdot]_Y}
&&(\mathscr M_{X,\mathsf I},{\rm d}_{\mathfrak Y}^{\rm LP})\ar@{ >->}'[d][dd]\ar@{ >->}'[r][rr]
&&(\overline{\mathscr M_{X,\mathsf I}},{\rm d}_{\mathfrak Y}^{\rm LP})\ar@{ >->}[dd]
\\
&({\rm Im}\,P(\mathsf I)/\!\equiv_Y,{\rm dist}_{\mathfrak Y})\ar@{ >->>}[ur]\ar@{ >->}[dd]\ar@{ >->}[rr]
&&(\overline{{\rm Im}\,P(\mathsf I)}^Y\!\!,{\rm dist}_{\mathfrak Y})\ar@{ >->>}[ur]\ar@{->>}[dd]
&&&\\
({\rm Rel}(\sigma),{\rm dist}_{\mathfrak Y})\ar@{->>}'[r][rr]^(.3){T_\sigma^Y}\ar@{->>}[dr]^{[\cdot]_Y}
&&({\mathscr M_\sigma^Y},{\rm d}_{\mathfrak Y}^{\rm LP})\ar@{ >->}'[r][rr]
&&(\overline{\mathscr M_\sigma^Y},{\rm d}_{\mathfrak Y}^{\rm LP})\\
&({\rm Rel}(\sigma)/\!\equiv_Y,{\rm dist}_{\mathfrak Y})\ar@{ >->>}[ur]\ar@{ >->}[rr]
&&(\overline{{\rm Rel}(\sigma)}^Y\!\!,{\rm dist}_{\mathfrak Y})\ar@{ >->>}[ur]
&&&\\
}\endxy
$$
\caption{The considered (pseudo)metric spaces and their relations. Unlabeled arrows correspond to inclusions ($\protect\xymatrix@C=20pt{*{}\protect\ar@{ >->}[r]&*{}}\ $) or isometric embeddings ($\protect\xymatrix@C=20pt{*{}\protect\ar@{ >->>}[r]&*{}}\ $). In this diagram the space $(\overline{{\rm Im}\,P(\mathsf I)}^Y\!\!,{\rm dist}_{\mathfrak Y})$ is the completion of the pseudometric space $({\rm Im}\,P(\mathsf I),{\rm dist}_{\mathfrak Y})$.}
\label{fig:diag}
\end{figure}

\subsection{Lift-Hausdorff convergence}
\label{sec:LHC}
We now show how all the above constructions 
 will 
 nicely fit in Definition~\ref{def:lHc} of the lift-Hausdorff convergence. We first show how the definition derives from the preceding notions dealing with general basic interpretations.

Let $\mathsf I:\tau\rightarrow\sigma$ be a fixed interpretation, let $X$ be a fixed fragment of ${\rm FO}(\tau)$, let $\mathfrak X$ be a fixed cover chain of $X$, and let $\mathfrak Y$ be any chain covering of $Y$ such that every formula in $L(\mathsf I)(Y_n)$ is logically equivalent to a formula in $X_n$.

According to Lemma~\ref{lem:contpush} the pushforward mapping $F(\mathsf I)_*$ is a continuous function from $P(S^X_\tau)$ to $P(S_\sigma^Y)$.
Then the L\'evy-Prokhorov distance ${\rm d}_{\mathfrak X}^{\rm LP}$ on $\overline{\mathscr M^X_\tau}$ defines a Hausdorff distance ${\rm d}_{\mathfrak X}^{\rm H}$ on the space $\mathscr C_{\overline{\mathscr M^X_\tau}}$ of non-empty closed subsets of $\overline{\mathscr M^X_\tau}$:
\begin{great}
	\begin{equation}
		{\rm d}_{\mathfrak X}^{\rm H}(M_1,M_2)=\max\Bigl(
		\sup_{\mu_1\in M_1}\inf_{\mu_2\in M_2}{\rm d}_{\mathfrak X}^{\rm LP}(\mu_1,\mu_2),\sup_{\mu_2\in M_2}\inf_{\mu_1\in M_1}{\rm d}_{\mathfrak X}^{\rm LP}(\mu_1,\mu_2)
		\Bigr).
	\end{equation}
\end{great}
 
$$
\xy\xymatrix@C=.1\textwidth{
&(\overline{\mathscr M^X_\tau},{\rm d}_{\mathfrak X}^{\rm LP})\ar@{->>}_-{F(\mathsf I)_*}[d]
&(\mathscr C_{\overline{\mathscr M^X_\tau}},{\rm d}_{\mathfrak X}^{\rm H})\\
&\overline{\mathscr M_{X,\mathsf I}},{\rm d}_{\mathfrak Y}^{\rm LP})\ar@{ >->}[d]\ar@{ >->}^-{\iota}[r]\ar@{ >->}^-{F(\mathsf I)_*^{-1}}[ur]
&(\mathscr C_{\mathscr M_{X,\mathsf I}},{\rm d}_{\mathfrak Y}^{\rm H})\ar@{ >->}_-{\widehat{F(\mathsf I)_*}^{-1}}[u]\\
&(\overline{\mathscr M_\sigma^Y},{\rm d}_{\mathfrak Y}^{\rm LP})\\
}\endxy
$$

Also the pseudometric ${\rm dist}_{\mathfrak X}$ on ${\rm Rel}(\tau)$ defines a Hausdorff pseudometric ${\rm dist}_{\mathfrak X}^{\rm H}$ on the space of non-empty closed subsets of ${\rm Rel}(\tau)$ (for the topology induced by the pseudometric ${\rm dist}_{\mathfrak X}$):

\begin{great}
	\begin{equation}
		{\rm dist}_{\mathfrak X}^{\rm H}(F_1,F_2)=\max\Bigl(
		\sup_{\mathbf A_1\in F_1}\inf_{\mathbf A_2\in F_2}{\rm dist}_{\mathfrak X}(\mathbf A_1,\mathbf A_2),\sup_{\mathbf A_2\in F_2}\inf_{\mathbf A_1\in F_1}{\rm dist}_{\mathfrak X}(\mathbf A_1,\mathbf A_2)
		\Bigr).
	\end{equation}
\end{great}

These (pseudo)metrics are related by the following equation (where $F_1$ and $F_2$ denote non-empty closed subsets of ${\rm Rel}(\tau)$:
	\[
{\rm dist}_{\mathfrak X}^{\rm H}(F_1,F_2)={\rm d}_{\mathfrak X}^{\rm H}(T_\tau^X(F_1),T_\tau^X(F_2))
\]

Using the injective mapping $F(\mathsf I)_*^{-1}$ we can transfer
to $\mathscr M_{X,\mathsf I}$ 
 the Hausdorff distance
 ${\rm d}_{\mathfrak X}^{\rm H}$ defined on $\mathscr C_{\overline{\mathscr M^X_\tau}}$, thus defining a
 distance ${\rm d}_{\mathfrak X,\mathsf I}^{\rm H}$ on $\mathscr M_{X,\mathsf I}$:
 \[
 {\rm d}_{\mathfrak X,\mathsf I}^{\rm H}(\mu_1,\mu_2)={\rm d}_{\mathfrak X}^{\rm H}(F(\mathsf I)_*^{-1}(\mu_1),F(\mathsf I)_*^{-1}(\mu_2)).
 \]
 (Note that this metric usually does not define the same topology as $d^{\rm LP}_{\mathfrak Y}$.)
 
Using the mapping $T_\sigma^Y$ we can transfer to  ${\rm Im}\,P(\mathsf I)$ the metric ${\rm dist}_{\mathfrak X,\mathsf I}^{\rm H}$ just defined on $\mathscr M_{X,\mathsf I}$. As $T_\sigma^Y$ is not injective in general we get this way a pseudometric 
 ${\rm dist}_{\mathfrak X,\mathsf I}^{\rm H}$ on ${\rm Im}\,P(\mathsf I)$:
 \begin{align*}
 {\rm dist}_{\mathfrak X,\mathsf I}^{\rm H}(\mathbf A,\mathbf B)&={\rm d}_{\mathfrak X}^{\rm H}(F(\mathsf I)_*^{-1}(T_\sigma^Y(\mathbf A)),F(\mathsf I)_*^{-1}(T_\sigma^Y(\mathbf B)))\\
 	&={\rm d}_{\mathfrak X}^{\rm H}(T_\tau^X(P(\mathsf I)^{-1}(\mathbf A)),T_\tau^X(P(\mathsf I)^{-1}(\mathbf B))
 \end{align*}
 Hence we have
 \begin{great}
 	\begin{equation}
 	\label{eq:distHdH}
 	{\rm dist}_{\mathfrak X,\mathsf I}^{\rm H}(\mathbf A,\mathbf B)={\rm dist}_{\mathfrak X}^{\rm H}(P(\mathsf I)^{-1}(\mathbf A),P(\mathsf I)^{-1}(\mathbf B)).
 	\end{equation}
 \end{great}
 
 The situation is summarized in the following diagram:
 
$$
\xy\xymatrix@C=.1\textwidth{
&(\mathscr C_{\overline{\mathscr M^X_\tau}},{\rm d}_{\mathfrak X}^{\rm H})
&(\mathscr C_{\overline{{\rm Rel}(\tau)}},{\rm dist}_{\mathfrak X}^{\rm H})\ar[l]_{T_\tau^X}
\\
&(\overline{\mathscr M_{X,\mathsf I}},{\rm d}_{\mathfrak X,\mathsf I}^{\rm H})\ar[u]_{F(\mathsf I)_*^{-1}}
&({\rm Im}\,P(\mathsf I),{\rm dist}_{\mathfrak X,\mathsf I}^{\rm H})\ar[u]_{P(\mathsf I)^{-1}}\ar[l]_{T_\sigma^Y}\\
}\endxy
$$

It follows from Lemma~\ref{lem:sm} that for every $\mathbf A,\mathbf B\in{\rm Im}\,P(\mathsf I)$ we have
\begin{great}
\begin{equation}
	{\rm dist}_{\mathfrak X,\mathsf I}^{\rm H}(\mathbf A,\mathbf B)\geq {\rm dist}_{\mathfrak Y}(\mathbf A,\mathbf B).
\end{equation}
\end{great}
(In particular the topology defined by the pseudometric ${\rm dist}_{\mathfrak X,\mathsf I}$ is finer that the topology defined by the pseudometric ${\rm dist}_{\mathfrak Y}$.)

Our basic notion of convergence with respect to an interpretation is the following (which we sketched in the introduction):

\begin{great}
	\begin{definition}[Lift-Hausdorff convergence]
	\label{def:lHc}
	Let $\mathsf I:\tau\rightarrow\sigma$ be a basic interpretation and let $X$ be a fragment of ${\rm FO}(\tau)$. A sequence
	$\seq{A}$ of finite $\sigma$-structures in ${\rm Im}\,P(\mathsf I)$	is {\em $\mathsf I_*(X)$-convergent} if, for every $X$-convergent subsequence $\seq{B}_f$ of lifts of $\seq{A}$ (meaning $P(\mathsf I)(\seq{B}_f)=\seq{A}_f$) there exists an $X$-convergent sequence $\seq{C}$ of lifts of $\seq{A}$ extending $\seq{B}_f$ (i.e. such that $P(\mathsf I)(\seq{C})=\seq{A}$ and $\seq{C}_f=\seq{B}_f$).

	We  refer to the general notion of $\mathsf I_*(X)$-convergence as {\em lift-Hausdorff convergence} (or simply {\em lift convergence}).
\end{definition}
\end{great}

This convergence admits an alternative equivalent definition, which justifies the term of ``lift-Hausdorff convergence'':

\begin{theorem}[Metrization]
\label{thm:Cauchy}
	Let $\mathfrak X$ be an arbitrary  cover chain of $X$, and let $\mathfrak Y$ be an arbitrary chain covering of $Y=L(\mathsf I)^{-1}(X)$ such that every formula in $L(\mathsf I)(Y_n)$ is logically equivalent to a formula in $X_n$.

	Then a sequence $\seq{A}$ of $\sigma$-structures in ${\rm Im}\,P(\mathsf I)$ is $\mathsf I_*(X)$-convergent if and only if it is Cauchy for ${\rm dist}_{\mathfrak X,\mathsf I}^{\rm H}$.
	
\end{theorem}
\begin{proof}
We consider the two implications.

First assume that the sequence $\seq{A}$ of $\sigma$-structures in ${\rm Im}\,P(\mathsf I)$ is $\mathsf I_*(X)$-convergent and let $\seq{B}_f$ be an $X$-convergent subsequence such that $P(\mathsf I)(\seq{B}_f)=\seq{A}_f$. 
For every positive integer $m$, let $N(m)$ be minimum integer such that
 $f(N(m))\geq m$. Let $\mathbf C_m$ be a $\tau$-structure in 
 $P(\mathsf I)^{-1}(\mathbf A_{N(m)})$ such that 
 ${\rm dist}_{\mathfrak X}(\mathbf C_m,\mathbf B_{f(N(m))})$ is minimum. Note that the minimum is attained as $P(\mathsf I)^{-1}(\mathbf A_{N(m)})$ is compact. By definition we have 
 \[
 {\rm dist}_{\mathfrak X}(\mathbf C_m,\mathbf B_{f(N(m))})\leq {\rm dist}_{\mathfrak X,\mathsf I}(\mathbf A_m,\mathbf A_{f(N(m))}).
 \]
 As $\seq{A}$ is Cauchy for ${\rm dist}_{\mathfrak X,\mathsf I}$ and $\seq{B}_f$ is Cauchy for ${\rm dist}_{\mathfrak X}$  it directly follows that $\seq{C}$ is Cauchy for ${\rm dist}_{\mathfrak X}$, i.e. that $\seq{C}$ is $X$-convergent.

We now consider the other direction.
Assume that for every $X$-convergent subsequence $\seq{B}_f$ such that $P(\mathsf I)(\seq{B}_f)=\seq{A}_f$ there exists a sequence $\seq{C}$ such that $\seq{C}_f=\seq{B}_f$ and $P(\mathsf I)(\seq{C})=\seq{A}$, and
assume for contradiction that the sequence $\seq{A}$ is not $\mathsf I_*(X)$-convergent. Then there exists $\alpha>0$, such  that for every integer $N$ there exist integers $n,m>N$ and 
$\mathbf{B}_m\in P(\mathsf I)^{-1}(\mathbf A_m)$ such that for every 
$\mathbf{C}_n\in P(\mathsf I)^{-1}(\mathbf A_n)$ we
have 
${\rm dist}_{\mathfrak X}(\mathbf B_m,\mathbf C_n)>\alpha$.
This allows to construct subsequence $\seq{B}_f$ and $\seq{C}_g$ (where $(f(i),g(i)$ correspond to a pair of admissible values of $m$ and $n$ with $\min(m,n)>\max(f(i-1),g(i-1))$. Moreover, we can assume that $\seq{B}_f$ is $X$-convergent. By assumption the subsequence $\seq{B}_f$ can be extended into a full $X$-convergent sequence, which we (still) denote by $\seq{B}$ such that $P(\mathsf I)(\seq{B})=\seq{A}$. In particular, there exist some $N$ such that for every $n,m>N$ we have
${\rm dist}_{\mathfrak X}(\mathbf B_n,\mathbf B_m)<\alpha$.
In particular, 
${\rm dist}_{\mathfrak X}(\mathbf B_{f(n)}\mathbf B_{g(n)})<\alpha$, 
what contradicts the minimality hypothesis on
${\rm dist}_{\mathfrak X}(\mathbf B_{f(n)},\mathbf C_{g(n)})$.
\end{proof}

Note that Theorem~\ref{thm:Cauchy} clearly states that the property of a sequence to be Cauchy for ${\rm dist}_{\mathfrak X,\mathsf I}^{\rm H}$ is independent of the particular choice of the chain coverings $\mathfrak X$ and $\mathfrak Y$.

\begin{theorem}[Representation]
\label{thm:LGrep}
	The $\mathsf I_*(X)$-limit of a sequence of an $\mathsf I_*(X)$-convergent sequence can be uniquely represented by means of a non-empty compact subset of $\overline{\mathscr M^X_\tau}$.
\end{theorem}
\begin{proof}
	Let $\seq{A}$ be an $\mathsf I_*(X)$-convergent sequence.
	Let $Y=L(\mathsf I)^{-1}(X)$. We fix a cover chain $\mathfrak X$ of $X$ and a cover chain $\mathfrak Y$ of $Y$ such that every formula in $L(\mathsf I)(Y_n)$ is logically equivalent to a formula in $X_n$.
	
	According to Lemma~\ref{lem:contpush} the pushforward mapping $F(\mathsf I)_*$ is a continuous function from $P(S^X_\tau)$ to $P(S_\sigma^Y)$ hence for every $\mathbf A$ the set 
	\[
	\mathscr L(\mathbf A)=\{T_\tau^X(\mathbf A^+)\mid P(\mathsf I)(\mathbf A^+)=\mathbf A\}
	\]
	is a non empty closed (hence compact) subset of $\overline{\mathscr M^X_\tau}$. According to Theorem~\ref{thm:Cauchy}, the $\mathsf I_*(X)$-convergence of $\seq{A}$ is equivalent to the convergence of $\seq{A}$ according to ${\rm dist}_{\mathfrak X,\mathsf I}^{\rm H}$ metric. As noticed in beginning of Section~\ref{sec:LHC} we have
	 \begin{align*}
 {\rm dist}_{\mathfrak X,\mathsf I}^{\rm H}(\mathbf A,\mathbf B)&={\rm d}_{\mathfrak X}^{\rm H}(T_\tau^X(P(\mathsf I)^{-1}(\mathbf A)),T_\tau^X(P(\mathsf I)^{-1}(\mathbf B))\\
 &={\rm d}_{\mathfrak X}^{\rm H}(\mathscr L(\mathbf A),\mathscr L(\mathbf B))
 \end{align*}
Thus, as ${\rm d}_{\mathfrak X}^{\rm H}$ is the Hausdorff distance on the space $\mathscr C_{\overline{\mathscr M^X_\tau}}$ of non-empty closed subsets of $\overline{\mathscr M^X_\tau}$ defined by the L\'evy-Prokhorov distance ${\rm d}_{\mathfrak X}^{\rm LP}$ on $\overline{\mathscr M^X_\tau}$, the $\mathsf I_*(X)$-convergence of $\seq{A}$ is equivalent to the convergence of the sequence
	$\mathscr L(\seq{A})=(\mathscr L(\mathbf A_n))_{n\in\mathbb N}$ in the Hausdorff sense.
	
	It follows that the limit of $\seq{A}$ can be represented uniquely by the Hausdorff limit of $\mathscr L(\seq{A})$, which is a non-empty compact subset of $\overline{\mathscr M^X_\tau}$.
\end{proof}

This lemma gives an easy proof of the following result.
\begin{proposition}
\label{prop:XY}
Let $\mathcal C$ be a class of $\sigma$ structures, let $\mathsf I:\tau\rightarrow\sigma$ be an interpretation, and let $X,Y$ be fragments of ${\rm FO}(\tau)$.

	If $X$-convergence implies $Y$-convergence in the class  $\mathcal D=\{\mathbf B\mid P(\mathsf I)(\mathbf B)\in\mathcal C\}$ of $\tau$-structures then 
	$\mathsf I_*(X)$-convergence implies $\mathsf I_*(Y)$-convergence in the class $\mathcal C$.
\end{proposition}
\begin{proof}
	Let $\seq{A}$ be an $\mathsf I_*(X)$-convergent sequence of $\tau$-structures in $\mathcal C$ and let $\seq{B}_f$ be a $Y$-converging subsequence of $\tau$-structures (in $\mathcal D$) such that $P(\mathsf I)(\mathbf B_{f(n)})=\mathbf A_{f(n)}$.
		Let $\seq{B}_{g\circ f}$ be an $X$-converging subsequence of $\seq{B}_f$. Then there exists, according to Theorem~\ref{thm:SC} an $X$-convergent sequence $\seq{C}$ such that $\seq{C}_{g\circ f}=\seq{B}_{g\circ f}$ and
		$P(\mathsf I)(\seq{C})=\seq{A}$ (hence $\seq{C}$ is in $\mathcal D$). As $X$-convergence implies $Y$-convergence on $\mathcal D$ the sequence $\seq{D}$ is $Y$ convergent, and has the same $Y$-limit as the $Y$-convergent sequence $\seq{B_f}$ as they share infinitely many elements. It follows that the sequence $\seq{D}$ defined by 
\[
\mathbf D_n=\begin{cases}
\mathbf B_{n}&\text{if $(\exists i)\ n=f(i)$}\\
\mathbf C_n&\text{otherwise}
\end{cases}
\]
as the property that $\seq{D}_f=\seq{B}_f$ and $P(\mathsf I)(\seq{D})=\seq{A}$. By Theorem~\ref{thm:SC} we deduce that $\seq{A}$ is $\mathsf I_*(Y)$-convergent.
\end{proof}

Here are some more remarks indicating convenient properties of $\mathsf I_*(X)$-convergence.

First note that if $\mathsf I:\sigma\rightarrow\sigma$ is the identity interpretation, then
${\rm dist}_{\mathfrak X,\mathsf I}={\rm dist}_{\mathfrak X}$ and 
 $\mathsf I_*(X)$-convergence is the same as $X$-convergence.
Also, we have that every sequence $\seq{A}$ in ${\rm Im}\, P(\mathsf I)$ has an $\mathsf I_*(X)$-convergent subsequence.
Finally, let us remark
	that for every $\mathsf I:\tau\rightarrow\sigma$, $\mathsf I_*({\rm FO})$-convergence implies ${\rm FO}$-convergence. 

Let $\hat\tau$ be the signature obtained from $\tau$ by duplicating each relation symbol countably many times, which we denote by
 $\hat\tau=\mathbb N\tau$. To each symbol $R\in\tau$ correspond the symbols $R^i$ in $\hat\tau$ (for $i\in\mathbb N$).
We define the interpretation $\mathsf I_i$ obtained from $\mathsf I$ by replacing relations $R$ by $R^i$ ($\mathsf I_i$ is a clone of $\mathsf I$ based on the relations $R^i$).

\begin{proposition}[Almost $\mathsf I_*(X)$-limit probability measure]
	Let $\seq{A}$ be an $\mathsf I_*(X)$-convergent sequence of finite $\sigma$-structures.

	There exists a probability measure $\hat\mu\in \overline{\mathscr M_{\hat\tau}^X}$ such that
	for every $\epsilon>0$ and for every $\seq{C}$ such that $P(\mathsf I)(\seq{C})=\seq{A}$ there exists $i\in\mathbb N$ such that 
\[
{\rm d}^{LP}_{\mathfrak X}(F(\mathsf I_i)_*(\hat\mu), \lim_n T_{\tau}^X(\mathbf C_n))<\epsilon,
\]
where $\lim_n$ stands for the weak limit of probability measures.
\end{proposition}
\begin{proof}
For $i\in\mathbb N$ we choose $\mathbf B_n^i$ such that $P(\mathsf I)(\mathbf B_n^i)=\mathbf A_n$. We construct the $\hat\tau$-structure
$\hat{\mathbf B}_n$ by amalgamating all the relations of all the $\mathbf B_n^i$. We denote by $\mathsf S_i$ the interpreting projection $\hat{\mathbf B}_n\mapsto\mathbf B_n^i$. Note that $\mathsf I_i=\mathsf I\circ\mathsf S_i$. Then we have

$$\xy\xymatrix{
\text{$\hat\tau$-structures}\ar@{..}[rrr]&&&\hat{\mathbf B}_n\ar_{P(\mathsf S_1)}[dll]\ar^-{P(\mathsf S_2)}[dl]\ar^{P(\mathsf S_i)}[dr]\\
\text{$\tau$-structures}\ar@{..}[r]&\mathbf B_n^1\ar_{P(\mathsf I)}[drr]&\mathbf B_n^2\ar^-{P(\mathsf I)}[dr]&\cdots&\mathbf B_n^i\ar^{P(\mathsf I)}[dl]&\cdots\\
\text{$\sigma$-structures}\ar@{..}[rrr]&&&\mathbf A_n
}\endxy$$

Then we consider an $X$-convergent subsequence $\seq{\hat B}_f$ of $\seq{\hat B}$, the limit of which is represented by the probability measure $\hat\mu\in\mathscr M_{\hat\tau}^X$. The measure $\hat\mu$ has obviously the claimed property.
\end{proof}
	
\subsection{Local Global Convergence}
\label{sec:lg}
In this section we show how the abstract framework of Section~\ref{sec:LHC} provides a proper setting for local-global convergence.

The notion of local-global convergence of graphs with bounded degrees has been introduced by Bollob{\'a}s and Riordan \cite{bollobas2011sparse}  based on a colored neighborhood metric. In \cite{hatami2014limits}, Hatami, Lov\'asz, and Szegedy gave the following equivalent definition:

\begin{definition}[\cite{hatami2014limits}]
\label{def:LG-HL}
	A graph sequence $\seq{G}$ of graphs with maximum degree $D$ is {\em local-global convergent} if for every $r, k \in\mathbb N$ and $\epsilon > 0$ there is an index $l$ such that if $n, m > l$, then for every coloring of the vertices of $G_n$ with $k$ colors, there is a coloring of the vertices of $G_m$ with $k$ colors such that the total variation distance between the distributions of colored neighborhoods of radius $r$ in $G_n$ and $G_m$ is at most $\epsilon>0$.
\end{definition}

The following is the principal result which relates local-global convergence to a lift-Hausdorff convergence. 

Let us consider a fixed countable signature $\sigma$ and the signature $\tau$ obtained from $\sigma$ by adding countably many unary symbols. Thus $\sigma\subset\tau$.
Let 
\[\mathsf{Sh}:\sigma\rightarrow\tau\]
be the forgetful interpretation ($\mathsf{Sh}$ for ``Shadow''). This means $\mathsf{Sh}=(\nu,\eta,(\rho_R)_{R\in\sigma})$, where $\nu(x_1):=(x_1=x_1)$, $\eta(x_1,x_2):=(x_1=x_2)$, and 
$\rho_R(x_1,\dots,x_p):=R(x_1,\dots,x_p)$ for $R\in\sigma$ with arity $p$.
Then, for instance:
\begin{itemize}
	\item for a $\tau$-structure $\mathbf A$, the $\sigma$-structure $P(\mathsf{Sh})(\mathbf A)$ is obtained from $\mathbf A$ by forgetting all unary relations in $\tau\setminus\sigma$;
	\item for a formula $\phi\in {\rm FO}(\sigma)$, we have 
 $L(\mathsf{Sh})(\phi)=\phi$;
	\item for $t\in S_\tau$ we have we have $F(\mathsf{Sh})(t)=t\cap{\rm FO}(\sigma)$.
\end{itemize}

By \cite{CMUC} we 
know that ${\rm FO}^{\rm local}$-convergence coincides with ${\rm FO}^{\rm local}_1$-convergence for graphs with bounded degree. By Proposition~\ref{prop:XY} the notions of $\mathsf{Sh}_*({\rm FO}^{\rm local})$-convergence and $\mathsf{Sh}_*({\rm FO}^{\rm local}_1)$-convergence will also coincides for graphs with bounded degrees. These notions actually coincides with the notion of local-global convergence of graphs with bounded degrees:

\begin{proposition}
\label{prop:LG}
Let $\seq{G}$ be a sequence of graphs with maximum degree $D$.
Then the following are equivalent:
\begin{enumerate}
	\item $\seq{G}$ is local-global convergent,
	\item $\seq{G}$ is ${\mathsf{Sh}}_*({\rm FO}_1^{\rm local})$-convergent,
	\item $\seq{G}$ is ${\mathsf{Sh}}_*({\rm FO}^{\rm local})$-convergent.
\end{enumerate} 
\end{proposition}
\begin{proof}
For classes of colored graphs with degree at most $D$, ${\rm FO}^{\rm local}_1$-convergence is equivalent to ${\rm FO}^{\rm local}$-convergence (see \cite{CMUC}). It follows from Proposition~\ref{prop:XY} that for these graphs ${\mathsf{Sh}}_*({\rm FO}^{\rm local}_1)$-convergence is equivalent to ${\mathsf{Sh}}_*({\rm FO}^{\rm local})$-convergence. Thus we only have to prove the equivalence of local-global convergence and ${\mathsf{Sh}}_*({\rm FO}_1^{\rm local})$-convergence.

	We consider the fragment $X\subset {\rm FO}_1^{\rm local}$ of formulas consistent with the property of having maximum degree $D$.
	Consider a cover chain $\mathfrak X=(X_r)_{r\in\mathbb N}$ of $X$
where $X_r$ contains (one representative of the equivalence class of) each formula in $X$ that is $r$-local and use only the $r$ first unary predicates. (Note that $|X_r|$ is finite.)

	It is easily checked that every $r$-local formula $\phi\in X_r$ is equivalent (on graphs with maximum degree $D$) to a formula a the form
	$\bigvee_{B\in \mathcal F_\phi} \zeta_{B,r}(x)$ where $\zeta_{B,r}(x)$ expresses that the ball of radius $r$ rooted at $x$ is isomorphic to the rooted graph $B$, and $\mathcal F_\phi$ is a finite set of rooted graphs of radius at most $r$. It easily follows that the maximum of $|\langle\phi, G_1\rangle-\langle\phi,G_2\rangle|$ over $\phi\in X_r$ equals the total variation distance of the distributions of $r$-balls in $G_1$ and $G_2$ where we consider only the $r$ first colors, which we denote by ${\rm d}_{TV}^{(r)}(G_1,G_2)$. Then we have
\begin{equation}
\label{eq:TV}
{\rm dist}_{\mathfrak X}(G_1,G_2)=
\inf_{r\in\mathbb N}\Bigl\{\max\bigl(1/r,{\rm d}_{TV}^{(r)}(G_1,G_2)\bigr)\Bigr\}
\end{equation}
As one easily checks that
${\rm d}_{TV}(r')(G_1,G_2)\geq {\rm d}_{TV}(r)(G_1,G_2)$ if $r'\geq r$ we have
that for every fixed integer $r$ we have 
\begin{equation}
\label{eq:TV2}
\min\bigl(1/r,{\rm d}_{TV}^{(r)}(G_1,G_2))\bigr)\leq {\rm dist}_{\mathfrak X}(G_1,G_2)\leq \max\bigl(1/r, {\rm d}_{TV}^{(r)}(G_1,G_2)\bigr).
\end{equation}

	Now assume $\seq{G}$ is ${\mathsf{Sh}}_*({\rm FO}_1^{\rm local})$-convergent. Let $k,r$ be a fixed integer. Then ${\mathsf{Sh}}_*({\rm FO}_1^{\rm local})$-convergence of $\seq{G}$ easily implies the convergence of the lifts of $G_n$ by $k$ colors, which means that  for every $\epsilon > 0$ there is an index $l$ such that if $n, m > l$, then for every coloring $G_n^+$ of  $G_n$ with $k$ colors, there is a coloring $G_m^+$ of $G_m$ with $k$ colors such that ${\rm dist}_{\mathfrak X}(G_n^+,G_m^+)<\epsilon$ hence by~\eqref{eq:TV2}
	the total variation distance between the distributions of colored neighborhoods of radius $r$ in $G_n$ and $G_m$ is at most $\epsilon>0$, provided that $\epsilon<1/r$.
	Hence $\seq{G}$ is local-global convergent.

Assume $\seq{G}$ is local-global convergent. Then for every $\epsilon>0$, letting $r=\lceil 1/\epsilon\rceil$, 
 there exists an integer $l$ such that if $n, m > l$, then for every coloring $G_n^+$  of $G_n$ with $r$ colors, there is a coloring $G_m^+$ of  $G_m$ with $r$ colors such that the total variation distance between the distributions of colored neighborhoods of radius $r$ in $G_n$ and $G_m$ is at most $\epsilon$. Hence by~\eqref{eq:TV2} we have
 ${\rm dist}_{\mathfrak X}(G_n^+,G_m^+)<\max(\epsilon,{\rm d}_{TV}^{(r)}(G_n,G_m))\leq\epsilon$. (Note that we do not need to use any of the colors with index greater than $r$.) It follows that $\seq{G}$ is ${\mathsf{Sh}}_*({\rm FO}_1^{\rm local})$-convergent. 
 \end{proof}

Motivated by this theorem we can extend the definition of local-global convergence to general graphs and relational strcutures:

\begin{definition}[Local-global convergence]
\label{def:LG}
	A sequence $\seq{A}$ is {\em local-global convergent} if it is $\mathsf{Sh}_*({\rm FO}^{\rm local})$-convergent.
\end{definition}

The weaker notion of ${\mathsf{Sh}}_*({\rm FO}_1^{\rm local})$-convergence already implies convergence of some graph invariants in an interesting  way. This is, for instance, the case of the Hall ratio $\rho(G)=\alpha(G)/|G|$.
\begin{proposition}
Let $\seq{G}$ be an ${\mathsf{Sh}}_*({\rm FO}_1^{\rm local})$-convergent sequence of graphs.
The the Hall ratio $\rho(G_n)=\alpha(G_n)/|G_n|$ converges.
\end{proposition}
\begin{proof}
Let $a=\limsup\rho(G_n)$.  Let $G_n^+$ be obtained by
	 marking (by $M$) a maximum independent set in $G_n$.
	 (Thus $G_n=P(\mathsf{Sh})(G_n^+)$.)
	  We extract a subsequence of $\seq{G^+}$ with limit measure of $M(G_n^+)$ equal to $a$, then an ${\rm FO}_1^{\rm local}$-convergent subsequence.
	  	  According to the lifting property, this subsequence can be extended into a full sequence $\seq{G^*}$.Consider the formula 
	  	  \[
	  	  \psi(x):= M(x)\wedge (\exists y)(M(y)\wedge {\rm Adj}(x,y).
	  	  \]
	  	  Then $\psi(G_n^*)$ is the set of all marked vertices of $G_n$ with to a marked vertex in their neighborhood. Hence $\lim_{n\rightarrow\infty}\langle\psi, G_n^*\rangle=0$ (as it converges to $0$ on the subsequence where $M$ marks an independent set). Moreover, $M(G_n^*)\setminus\psi(G_n^*)$ is independent. It follows that
	  	  \[
	  	  a=\lim_{n\rightarrow\infty}\langle M,G_n^*\rangle= \lim_{n\rightarrow\infty}
	  	  \frac{|M(G_n^*)\setminus\psi(G_n^*)|}{|G_n^*|}
\leq \liminf \rho(G_n)\leq \limsup \rho(G_n)= a.
	  	  \]
	  	  Hence $\rho(G_n)$ converges.
\end{proof}

Let us add the following remarks:
In such a context it is not possible to distinguish (at the limit) a maximal independent set from a near maximal independent set. 
Of course this does not change the property that $\alpha(G_n)/|G_n|$ converges nor the measure of the (near) maximal independent set found in the limit. 

For the chromatic number, local-global convergence is clearly not strong enough, as witnessed by a local-global convergent sequence $\seq{G}$ of bipartite graphs modified by replacing $G_n$ by the disjoint union of $G_n$ and $K_{100}$ for (say) half of the values of $n$. The obtained sequence is still local-global convergent but the chromatic numbers of $G_n$ oscillate between $2$ and $100$. To ensure the convergence of the chromatic number one needs at least $\mathsf{Sh}({\rm FO}_0)$-convergence. However, with $\mathsf{Sh}({\rm FO}_1^{\rm local})$-convergence it is possible to get the convergence of the minimum integer $c$ such that the graphs $G_n$ can be made $c$-colorable by removing $o(|G_n|)$ vertices.

We end this section by giving an example showing that not every graphing is a strong local-global limit of a sequence of finite graphs (For a proof that not every graphing is a weak local-global limit of a sequence of finite graphs, that is an answer to the problem posed in \cite{hatami2014limits}, see \cite{Kun2018}).

\begin{example}
	Consider the graphing $\mathbf G$ with domain $(\mathbb R/\mathbb Z)\times \{1,2,\dots,6\}$, and edge set 

	\begin{align*}
E&=	\{\{(x,1),(y,1)\}, \{(x,6),(y,6)\}\mid x\in [0,1], y=x+\alpha\bmod 1\}\\
&	\cup \{\{(x,i),(x,i+1)\}\mid x\in [0,1], 1\leq i<5\}\\
&	\cup \{\{(x,2),(x,4)\}, \{(x,3),(x,5)\}\mid x\in [0,1]\}
	\end{align*}
	 represented on Fig.~\ref{fig:nonLG3}
\begin{figure}[ht]
	\begin{center}
		\includegraphics[width=.5\textwidth]{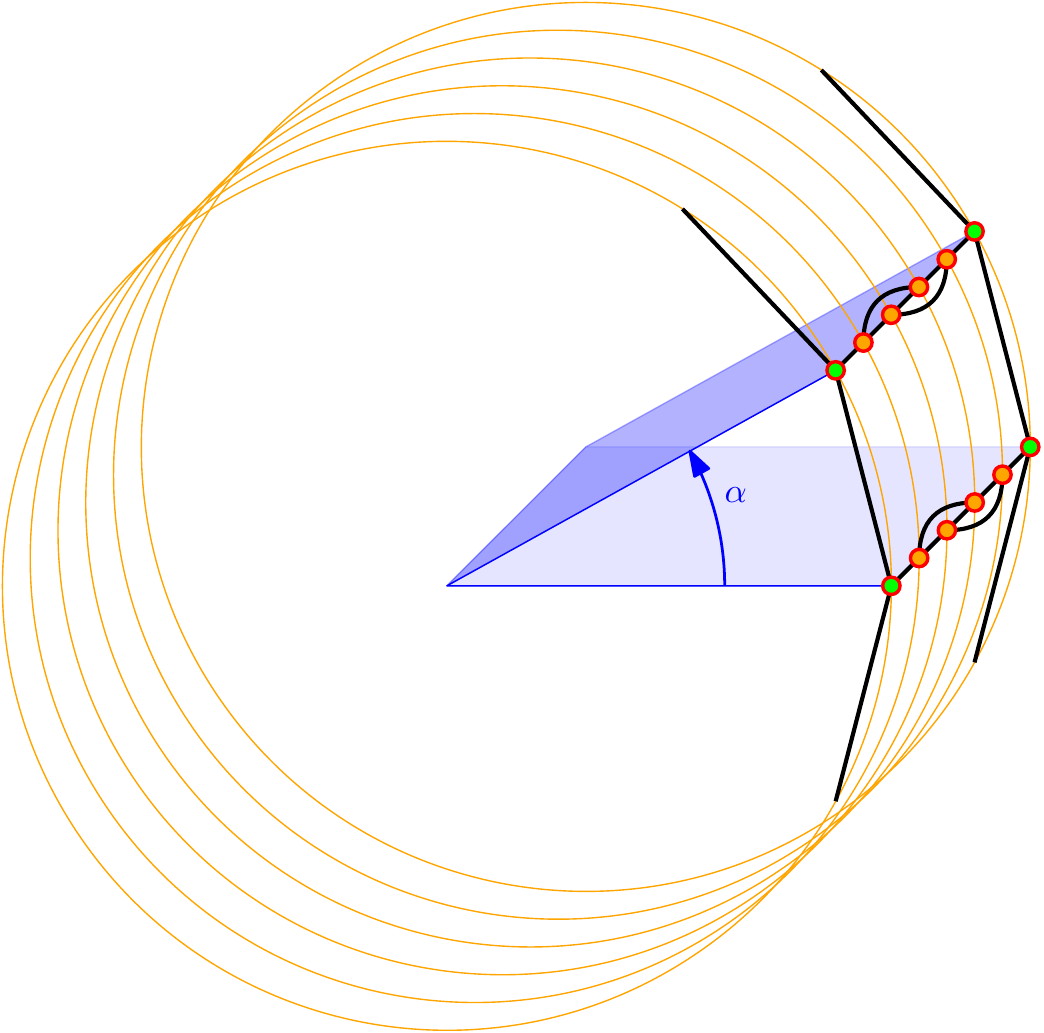}
	\end{center}
	\caption{Example of a $3$-regular graphing that is not a strong local global limit.}
	\label{fig:nonLG3}
\end{figure}

Assume $\mathbf G$ is the strong local-global limit of a	 sequence $\seq{G}=(G_n)_{n\in\mathbb N}$ of graphs. 
Almost all neighborhoods  in $G_n$ (for $n$ large) look the same: color red all vertices with two adjacent neighbors and color green all the vertices having exactly two non red neighbors. 

\begin{center}
	\includegraphics[width=.7\textwidth]{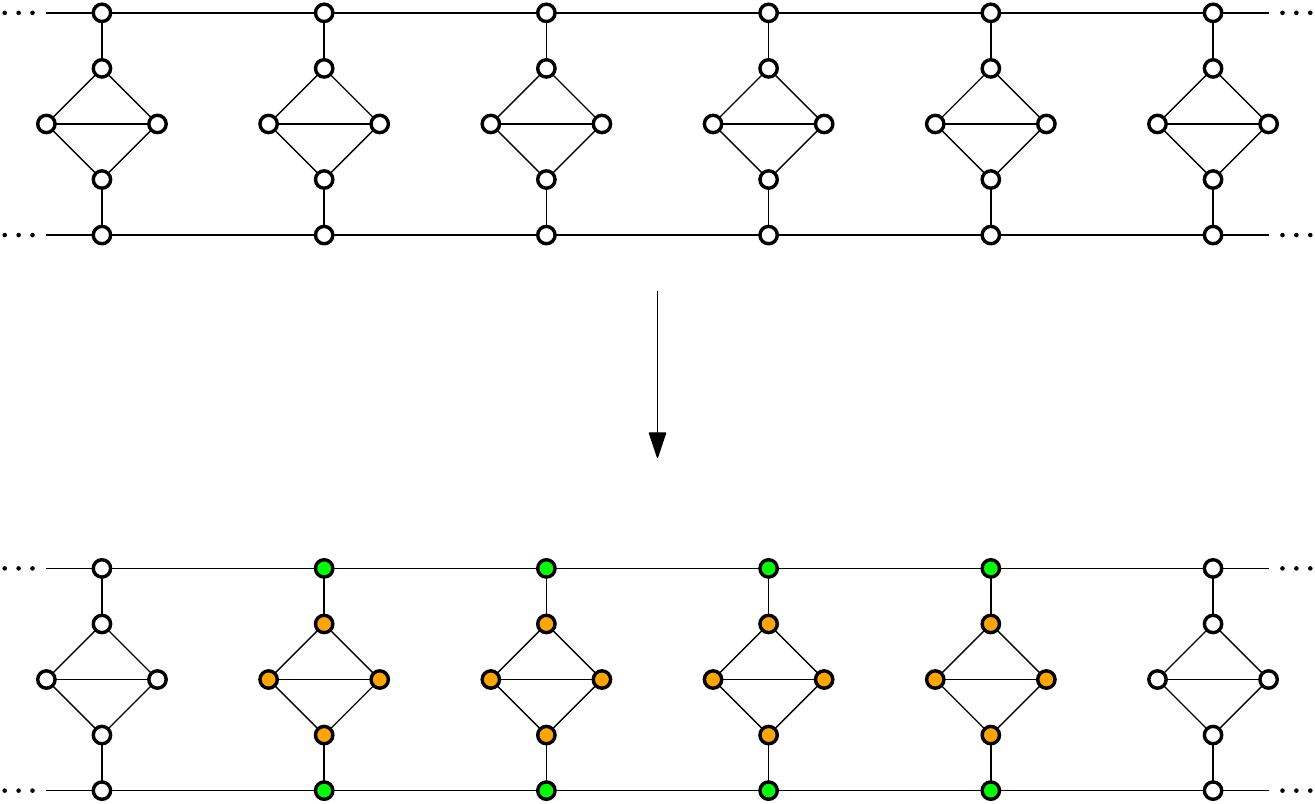}
\end{center}

Then, apart from a negligible set of vertices all the vertices are colored red and green. Moreover, almost all green vertices belong either to a long green cycle or a long green path. Recolor the green vertices in blue, green, and purple by dividing all these paths and cycles in almost equal parts (and taking care of globally balancing these colors).
	Now consider a local convergent subsequence of the colored $G_n$. By definition this local subsequence can be extended into a local convergent sequence $\seq{G^+}$ of colorings of the graphs in $\seq{G}$.
	By local convergence, every connected component of the sub-graphing induced by green, blue, and violet vertices are monochromatic (apart from a $0$ measure set), and these components are invariant by the transformation $(x,y)\mapsto (x+\alpha\bmod 1,y)$ (hence $y\in\{1,6\}$. However, this sub-graphing has two ergodic components, each of measure $1/6$ although each of the blue, green, and violet color contain asymptotically $1/9$ of the vertices, a contradiction.
\end{example}

\section{Applications}
\subsection{Clustering}
\label{sec:cluster}
Monadic lifts (i.e. lifts by unary relations) were considered in \cite{Loclim} in the context of continuous clustering of the structures in an ${\rm FO}^{\rm local}$-convergent sequence. One of the main results (see Theorem~\ref{thm:cluster} bellow) expresses that every ${\rm FO}^{\rm local}$-convergent sequence has monadic lift tracing components while preserving  ${\rm FO}^{\rm local}$-convergence. This will be refined in this section under the stronger assumption of $\mathsf{Sh}_*({\rm FO}^{\rm local})$-convergence (see Theorem~\ref{thm:newcluster}).

 \begin{figure}[ht]
 	\begin{center}
\includegraphics[width=\textwidth]{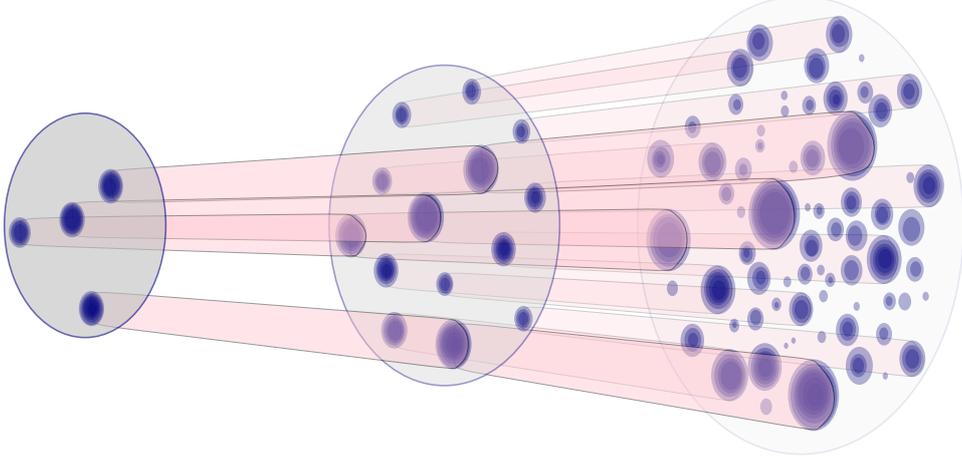}
\caption{Typical shape of a structure  continuously segmented by a clustering: dense spots correspond to globular clusters, and the background to the residual cluster.
Biggest globular clusters appear first and then move apart from each other, while new (smaller) globular clusters appear and residual cluster becomes sparser and sparser.}
\label{fig:milk}
 	\end{center}
 \end{figure}

The analysis in \cite{Loclim} 
leads to interesting notions:
  {\em globular cluster} (corresponding to a limit non-zero measure connected component), {\em residual cluster} (corresponding to all the zero-measure connected components taken as a whole), and {\em negligible cluster} (corresponding to the stretched part connecting the other clusters, which eventually disappears at the limit).

Negligible sets  intuitively  correspond to parts of the graph  one can  remove, without a great modification of the statistics of the graph: A sequence $\Xseq{X}\subseteq\Xseq{A}$ is
{\em negligible} in a local-convergent sequence $\seq{A}$ if
$$
\forall d\in\bbbn:\quad \limsup_{n\rightarrow\infty}\nu_{\mathbf A_n}(\ball[d]{\mathbf{A}_n}(X_n))=0.
$$
This we simply formulate as
$\forall d\in\bbbn:\ \limsup\nu_{\seq{A}}(\ball[d]{\seq{A}}(\Xseq{X}))=0$.

Two sequences $\Xseq{X}$ and $\Xseq{Y}$ of subsets are {\em equivalent}
in $\seq{A}$  if the sequence $\Xseq{X}\,\Delta\,\Xseq{Y}=(X_n\,\Delta\,Y_n)_{n\in\bbbn}$ is negligible in $\seq{A}$. This will be denoted by $\Xseq{X}\approx \Xseq{Y}$.
We denote by $\Xseq{0}$ the sequence of empty subsets. Hence $\Xseq{X}\approx\Xseq{0}$ is equivalent to the property that $\Xseq{X}$ is negligible.
We further define a partial order on sequences of subsets by $\Xseq{X}\preceq \Xseq{Y}$ if the sequence 
$\Xseq{X}\setminus \Xseq{Y}=(X_n\setminus Y_n)_{n\in\bbbn}$ is negligible in $\seq{A}$. Hence $\preceq$ has $\Xseq{0}$ for its minimum and $\Xseq{X}\approx\Xseq{Y}$ iff $\Xseq{X}\preceq \Xseq{Y}$ and $\Xseq{Y}\preceq \Xseq{X}$.

The notion of cluster of a local-convergent sequence  is a weak analog of the notion of union of connected components, or more precisely of the topological notion of ``clopen subset'':
A sequence $\Xseq{X}$ of subsets of a local-convergent sequence $\seq{A}$ is a 
{\em cluster} of  $\seq{A}$ if the following conditions hold:
\begin{enumerate}
\item the lifted sequence $\lift[\Xseq{X}]{\seq{A}}$ obtained by marking set $X_n$ in $\mathbf{A}_n$ by a new unary relation $M_{\Xseq{X}}$
is local-convergent;
\item the sequence $\partial_{\seq{A}}\Xseq{X}$ is negligible in $\seq{A}$.
\end{enumerate}

Condition (1) can be seen as a continuity requirement for the subset selection.
Condition (2) is stronger than the usual requirement that there are not too many connections leaving the cluster. We intuitively require that the (asymptotically arbitrarily large) ring around a cluster is a very sparse zone. 

A cluster $\Xseq{X}$ is
{\em atomic} if, for every cluster $\Xseq{Y}$ of $\seq{A}$ such that $\Xseq{Y}\preceq \Xseq{X}$ either $\Xseq{Y}\approx\Xseq{0}$ or $\Xseq{Y}\approx \Xseq{X}$; the cluster $\Xseq{X}$ is {\em strongly atomic} if $\Xseq{X}_f$ is an atomic cluster of $\seq{A}_f$ for every increasing function $f:\bbbn\rightarrow\bbbn$.
To the opposite, the cluster $\Xseq{X}$ is a {\em nebula}
if, for every increasing function $f:\bbbn\rightarrow\bbbn$, every atomic cluster $\Xseq{Y}_f$ of $\seq{A}_f$ with $\Xseq{Y}_f\subseteq \Xseq{X}_f$ is trivial (i.e. $\Xseq{Y}_f\approx\Xseq{0}$).
Finally,
a cluster $\Xseq{X}$ is {\em universal} for  $\seq{A}$ if $\Xseq{X}$ is a cluster of every conservative lift of $\seq{A}$.

Two clusters $\Xseq{X}$ and $\Xseq{Y}$ of a local-convergent
sequence $\seq{A}$ are {\em interweaving}, and we note
$\Xseq{X}\between \Xseq{Y}$ if 
every sequence $\Xseq{Z}$ with $Z_n\in\{X_n,Y_n\}$ is a cluster of 
$\seq{A}$.

	We say that two clusters $\Xseq{C_1}$ and $\Xseq{C_2}$ are 
	\begin{itemize}
		\item {\em weakly disjoint} if $\Xseq{C_1}\cap\Xseq{C_2}\approx\Xseq{0}$;
		\item {\em disjoint} if $\Xseq{C_1}\cap\Xseq{C_2}=\Xseq{0}$;
		\item {\em strongly disjoint} if $(\ball{\seq{A}}(\Xseq{C_1})\cap\Xseq{C_2})\cup(\Xseq{C_1}\cap\ball{\seq{A}}(\Xseq{C_2}))=\Xseq{0}$.
	\end{itemize} 

A cluster $\Xseq{C}$ of a
 local-convergent sequence $\seq{A}$ is {\em globular} if, for every $\epsilon>0$ there exists $d\in\bbbn$ such that
$$\adjustlimits\liminf_{n\rightarrow\infty}\sup_{v_n\in C_n}
\nu_{\mathbf A_n}(\ball[d]{\mathbf A_n}(v_n))>(1-\epsilon)\lim_{n\rightarrow\infty}\nu_{\mathbf A_n}(C_n).$$

In other words, a cluster $\Xseq{C}$ is globular if, for every $\epsilon>0$ and sufficiently large $n$, $\epsilon$-almost all elements of $C_n$ are included in some ball of radius at most $d$ in $C_n$, for some fixed $d$. (Note that for a cluster $\Xseq{C}$ and $v_n\in C_n$, considering $\nu_{\mathbf A_n}(\ball[d]{\mathbf A_n}(v_n))$ or $\nu_{\mathbf A_n}(\ball[d]{\mathbf A_n[C_n]}(v_n))$ makes asymptotically no difference.)
Every globular cluster is clearly strongly atomic, but the converse does not hold as witnessed, for instance, by sequence of expanders. The strongly atomic clusters that are not globular are called {\em open clusters}.
Opposite to globular clusters are residual clusters:
	A cluster $\Xseq{X}$ of $\seq{A}$ is {\em residual} if 
for every $d\in\bbbn$ it holds
$$
\adjustlimits\limsup_{n\rightarrow\infty}\sup_{v_n\in A_n}
\nu_{\mathbf A_n}(\ball[d]{\mathbf A_n}(v_n))=0.
$$

 \begin{theorem}[\cite{Loclim}]
 	\label{thm:cluster}
 	Let $\seq{A}$ be a local convergent sequence of $\sigma$-structures. Then there exists a signature
  $\sigma^+$ obtained from $\sigma$ by the addition of countably many unary symbols $M_R$ and $M_{i,j}$ ($i\in\bbbn$, $1\leq j\leq N_i$) and a clustering $\seq{A}^+$ of $\seq{A}$ with the following properties:
 	\begin{itemize}
 		\item For every $i\in\bbbn$, $\bigcup_{j=1}^{N_i}M_{i,j}(\seq{A})$  is a universal cluster; 
 		\item For every $i\in\bbbn$ and every $1\leq j\leq N_i$, $M_{i,j}(\seq{A})$  is a globular cluster; 
 		\item Two clusters $M_{i,j}(\seq{A})$ and $M_{i',j'}(\seq{A})$ are interweaving if and only  if $i=i'$;
 		\item $M_R(\seq{A})$  is a residual cluster.
 	\end{itemize}
 \end{theorem}	
 
This structural theorem is assuming the local convergence of the sequence. If we assume local-global convergence we get stronger results (Theorem~\ref{thm:newcluster} bellow) involving expanding properties which we will define now. This is pleasing as the decomposition into expanders was one of the motivating examples \cite{bollobas2011sparse} and \cite{hatami2014limits}.

The following is a sequential version of expansion property:
A structure $\mathbf{A}$ is {\em $(d,\epsilon,\delta)$-expanding} if, for every $X\subset A$ it holds
$$
\epsilon<\nu_{\mathbf A}(X)<1-\epsilon\quad\Longrightarrow\quad
\nu_{\mathbf A}(\ball[d]{\mathbf{A}}(X))>(1+\delta)\nu_{\mathbf A}(X).
$$
This condition may be reformulated as:
$$
\inf\biggl\{\frac{\nu_{\mathbf A}(\ball[d]{\mathbf{A}}(X)\setminus X)}{\nu_{\mathbf A}(X)}:\quad
\epsilon<\nu_{\mathbf A}(X)<1-\epsilon\biggr\}>\delta.
$$
Note that the left hand size of the above inequality
is similar to the {\em magnification} introduced in \cite{Alon1986}, which is the isoperimetric constant $h_{\rm out}$ defined by
$$
h_{\rm out}=\inf\left\{\frac{|\ball{\mathbf{A}}(X)\setminus X|}{|X|}:\quad 0<\frac{|X|}{|A|}<1/2\right\}.
$$

A local-convergent sequence $\seq{A}$ is {\em expanding}
if, for every $\epsilon>0$ there exist $d,t\in\bbbn$ and $\delta>0$ such that every $\mathbf{A}_n$ with $n\geq t$ is $(d,\epsilon,\delta)$-expanding.
A non-trivial cluster $\Xseq{X}$ of $\seq{A}$ is 
{\em expanding} 
of $\seq{A}$ if
$\seq{A}[\Xseq{X}]$ is expanding.
We have the following equivalent formulations of this concept:

\begin{lemma}[\cite{Loclim}]
\label{lem:satom}
Let $\Xseq{X}\not\approx\Xseq{0}$ be a cluster of a local convergent sequence $\seq{A}$.
The following conditions are equivalent:
\begin{enumerate}
\item  $\Xseq{X}$ is an expanding cluster of $\seq{A}$;
\item for every $\epsilon>0$ there exists $d,t\in\bbbn$ such that for every $\Xseq{Z}\subseteq\Xseq{X}$ with $\nu_{\seq{A}}(\Xseq{Z})>\epsilon\nu_{\seq{A}}(\Xseq{X})$ it holds 
$$\nu_{\seq{A}}(\ball[d]{\seq{A}}(\Xseq{Z}))>(1-\epsilon)\nu_{\seq{A}}(\Xseq{X});$$
\item the sequence $\Xseq{X}$ is a strongly atomic cluster of $\seq{A}$;
\item for every $\epsilon>0$ there exists no $\Xseq{Y}\subseteq\Xseq{X}$ such that $\partial_{\seq{A}}\Xseq{Y}\approx\Xseq{0}$ and
$$
\epsilon<\liminf \nu_{\seq{A}}(\Xseq{Y})<\lim \nu_{\seq{A}}(\Xseq{X})-\epsilon.
$$
\end{enumerate}
\end{lemma}

Note that for local-global convergent sequences, the notions of atomic, strongly atomic, and expanding clusters are equivalent.

The case of bounded degree graphs is particularly interesting and our definitions capture this as well.
Recall that a sequence $\seq{G}$ of graphs is a {\em vertex expander} if there exists  $\alpha>0$ such that 
$
\liminf h_{\rm out}(G_n)\geq\alpha
$.
(For more information on expanders we refer the reader to \cite{Hoory2006}.)

\begin{lemma}[\cite{Loclim}]
	Let $\seq{G}$ be a sequence of graphs with maximum degree at most $\Delta$ and let $\Xseq{C}\not\approx\Xseq{0}$ be a cluster of $\seq{G}$. The following are equivalent:
	\begin{itemize}
		\item $\Xseq{C}$ is an expanding cluster;
		\item for every $\epsilon>0$ there exists $\Xseq{X}\subseteq \Xseq{C}$ such that for every $n\in\bbbn$ it holds $|X_n|<\epsilon |C_n|$ and $\seq{G}[\Xseq{C}\setminus \Xseq{X}]$ is a vertex expander.
	\end{itemize}
\end{lemma}

We consider a fixed enumeration $\phi_1,\phi_2,\dots$ of ${\rm FO}^{\rm local}$.
The {\em profile} ${\rm Prof}(\Xseq{C})$ of a cluster $\Xseq{C}$ is the sequence formed by $\lim\nu_{\seq{A}}(\Xseq{C})$ followed by the values 
$\lim\langle\phi_i,\seq{A}[\Xseq{C}]\rangle$ for $i\in\mathbb N$. The lexicographic order on the profiles is denoted by $\leq$.

In \cite{Loclim} it was proved that two expanding clusters are either weakly disjoint or interweaving. We now prove a lemma with similar flavor.

\begin{lemma}
\label{lem:sa}
	Let $\Xseq{C}_1$ be an expanding cluster of a local-convergent sequence $\seq{A}$ and let $\Xseq{C}_2$ be a cluster of $\seq{A}$. 
	
	Then the limit set of  $\nu_{\seq{A}}(\Xseq{C}_1\cap\Xseq{C}_2)$ is included in $\{0,\lim \nu_{\seq{A}}(\Xseq{C}_1)\}$.
\end{lemma}
\begin{proof}
	Let $\Xseq{X}=\Xseq{C_1}\cap\Xseq{C_2}$.
	Assume for contradiction that  there exists $0<\alpha<\lim \nu_{\seq{A}}(\Xseq{C}_1)$ and a subsequence $\seq{A}_f$ such that $\lift[\Xseq{X}_f]{\seq{A}_f}$ is local convergent and $\lim\nu_{\seq{A}_f}(\Xseq{X}_f)=\alpha$. As $\delta_{\seq{A}}\Xseq{X}\subseteq \delta_{\seq{A}}\Xseq{C}_1\cup \delta_{\seq{A}}\Xseq{C}_2$ we deduce that $\Xseq{X}_f$ is a cluster of $\seq{A}_f$. But $\Xseq{X}_f\subseteq (\Xseq{C}_1)_f$, $\Xseq{X}_f\not\approx\Xseq{0}$, and $\Xseq{X}_f\not\approx\Xseq{C}_1$ (as 
	$\lim\nu_{\seq{A}_f}(\Xseq{X}_f)\notin\{0,\lim \nu_{\seq{A}}(\Xseq{C}_1)\}$), what contradicts the hypothesis that  $\Xseq{C}_1$ is expanding hence strongly atomic (see Lemma~\ref{lem:satom}).
\end{proof}

The following lemma is a restated version of a Lemma proved in \cite{Loclim}.
\begin{lemma}
	Two non-negligible clusters $\Xseq{C_1}$ and $\Xseq{C_2}$ are interweaving if and only if ${\rm Profile}(\Xseq{C_1})={\rm Profile}(\Xseq{C_2})$.\qed
\end{lemma}

Our main result in this section reveals the expanding structure of local-global convergent sequences.

\begin{theorem}
	\label{thm:newcluster}
	Let $\sigma$ be a countable relational signature, let $\sigma^+$ be the extension of $\sigma$ by countably many unary symbols $N$ and $U_i$ ($i\in\mathbb N$), and let $\sigma^*$ be the extension of $\sigma$ by countably many unary symbols $N$ and $M_{i,j}$.
	Let $\mathsf{I}:\sigma^*\rightarrow\sigma^+$ be the basic interpretation defined by $\rho_{U_i}(x):=\bigvee M_{i,j}(x)$ (and all other relations unchanged), 
	and let $\mathsf{Sh}^*:\sigma^*\rightarrow\sigma$ and
	$\mathsf{Sh}^+:\sigma^+\rightarrow\sigma$ be the natural forgetful interpretations.
	
	Then for every local-global convergent sequence $\seq{A}$ of $\sigma$-structure there exists a local-convergent sequence $\seq{A}^*$ such that
\begin{itemize}
	\item $P(\mathsf{Sh}^*)(\seq{A}^*)=\seq{A}$,
	\item for every $i,j\in\mathbb N$, $M_{i,j}(\seq{A}^*)$ is either null or an atomic cluster of $\seq{A}^*$, which is interweaving with $M_{i',j'}(\seq{A}^*)$  if and only  if $i=i'$,
	 \item $N(\seq{A}^*)$ is a nebula cluster of $\seq{A}^*$,
\end{itemize} 
and such that $\seq{A}^+=P(\mathsf{I})(\seq{A}^*)$ has the following properties:
\begin{itemize}	 
	\item $\seq{A}^+$ is a local-global convergent sequence such that $P(\mathsf{Sh}^+)(\seq{A}^+)=\seq{A}$,
	\item for every $i\in\mathbb N$, $U_{i}(\seq{A}^+)$ is either null or a cluster of $\seq{A}^+$, which can be covered by (finitely many) interweaving atomic clusters,
	 \item $N(\seq{A}^+)$ is a nebula cluster of $\seq{A}^+$.
\end{itemize} 
\end{theorem}

Note that this result is in agreement with the intuition: The $\sigma^*$-lift is ``finer'' than the $\sigma^+$-lift and thus less likely to be local-global convergent. For instance, we can refer to the subsequence extension property stated in the definition of lift-Hausdorff convergence (Definition~\ref{def:lHc}).
\begin{proof}
Let $\seq{A}$ be a local-global convergent sequence.
We select inductively clusters expanding clusters $\Xseq{C}^{i,j}$ of $\seq{A}$ as follows: We start with $\Xseq{Z}=\Xseq{0}$, $i=1$, $j=1$ and let $P$ be the maximum profile of an expanding cluster of $\seq{A}$.
Then we repeat the following procedure as long as there exists an expanding cluster of $\seq{A}$ that is weakly disjoint from $\Xseq{Z}$
\begin{itemize}
	\item If there exists an expanding cluster in $\seq{A}$ with profile $P$ that is weakly disjoint from $\Xseq{Z}$ we select one as $\Xseq{C}^{i,j}$, 
	we let $\Xseq{Z}\leftarrow\Xseq{Z}\cup \Xseq{C}^{i,j}$, and we increase $j$ by $1$.
	\item Otherwise, we select one with maximum profile as $\Xseq{C}^{i+1,1}$, e let $\Xseq{Z}\leftarrow\Xseq{Z}\cup \Xseq{C}^{i+1,1}$, we let $P$ be the profile of $\Xseq{C}^{i+1,1}$, we increase $i$ by $1$, and let $j=1$.
\end{itemize}
It is easily checked that by modifying marginally the clusters $\Xseq{C}^{i,j}$ we can make them disjoint and such that $\Xseq{N}=\Xseq{A}\setminus\bigcup_{i,j}\Xseq{C}^{i,j}$ is a nebula cluster. Then by \cite[Corollary 5]{Loclim} lifting $\seq{A}$ by marking $M_{i,j}$ the cluster $\Xseq{C}^{i,j}$ and $N$ the cluster $\Xseq{N}$ we get a local-convergent sequence $\seq{A}^*$, which obviously satisfies the conditions stated in the Theorem.

Let $\seq{A}^+=P(\mathsf{I})(\seq{A}^*)$. The only property we still have to prove is that $\seq{A}^+$ is local-global convergent. According to Definition~\ref{def:lHc} this boils down to proving that every local-convergent subsequence $\seq{B}^+_f$ of lifts of $\seq{A}^+$ can be extended into a full local-convergent sequence of lifts of $\seq{A}^+$. We can transfer the relations $M_{i,j}$ from $\seq{A}^*$ to $\seq{B}^+_f$. This way we obtain a subsequence $\seq{B}^*_f$ of lifts of $\seq{A}^*$ (which does not need to be local convergent), such that $P(\mathsf{I})(\seq{B}^*_f)=\seq{B}^+_f$. Let $\seq{B}^*_{g\circ f}$ be a local-convergent subsequence of $\seq{B}^*_f$. As $P(\mathsf{Sh}^*)(\seq{B}^*_{g\circ f})=\seq{A}_{g\circ f}$ and $\seq{A}$ is local-global convergent there exists a local-convergent sequence $\seq{D}^*$ of lifts of $\seq{A}$ extending $\seq{B}^*_{g\circ f}$, that is:  $P(\mathsf{Sh}^*)(\seq{D}^*)=\seq{A}$ and  $\seq{D}^*_{g\circ f}=\seq{B}^*_{g\circ f}$. Let $\hat{\Xseq{C}}^{i,j}=M_{i,j}(\seq{D}^*)$. As $(\hat{\Xseq{C}}^{i,j})_{g\circ f}=(\Xseq{C}^{i,j})_{g\circ f}$ we get that $\hat{\Xseq{C}}^{i,j}$ is a cluster of $\seq{A}$ with same profile as $\Xseq{C}^{i,j}$. According to Lemma~\ref{lem:sa}, the limit set of $\nu_{\seq{A}}(\hat{\Xseq{C}}^{i,j}\cap \Xseq{C}^{i,j'})$ is included in $\{0,m\}$, where $m=\lim\nu_{\seq{A}}(\hat{\Xseq{C}}^{i,j})=\lim\nu_{\seq{A}}(\Xseq{C}^{i,j'})$. It follows that either $\hat{\Xseq{C}}^{i,j}\preceq\bigcup_{j'}\Xseq{C}^{i,j'}$ or there exists a subsequence $\seq{A}_h$ of $\seq{A}$ such that $(\hat{\Xseq{C}}^{i,j})_h$ is weakly disjoint from the cluster $(\bigcup_{j'}\Xseq{C}^{i,j'})_g$. Marking all the clusters $\Xseq{C}^{i,j}$ and $\hat{\Xseq{C}}^{i,j}$ in $\seq{A}_h$ we get a local-convergent subsequence of lifts, which can be extended into full local-convergent sequence of lifts of $\seq{A}$. In this sequence, the marks corresponding to the extension of $(\hat{\Xseq{C}}^{i,j})_g$ will correspond to a cluster of $\seq{A}$ disjoint from all the clusters $\Xseq{C}^{i,j'}$ but with the same profile, which contradicts the construction procedure of the clusters $\Xseq{C}^{i,j}$. Thus $\hat{\Xseq{C}}^{i,j}\preceq\bigcup_{j'}\Xseq{C}^{i,j'}$, and $\bigcup_{j}\hat{\Xseq{C}}^{i,j}\preceq\bigcup_{j'}\Xseq{C}^{i,j'}$. As these two clusters have same limit measure we have $\bigcup_{j}\hat{\Xseq{C}}^{i,j}\approx\bigcup_{j'}\Xseq{C}^{i,j'}$. This means that $P(\mathsf I)(\seq{D}^*)$ and $P(\mathsf I)(\seq{B}^*)$ are sufficiently close, so that if we consider the lifts of $\seq{A}^+$ defined by $\seq{B}^+$ for indices of the form $f(n)$ for some $n\in\mathbb N$ and by $P(\mathsf I)(\seq{D}^*)$ for the other indices, we get a local-convergent sequence of lifts of $\seq{A}^+$ which extends $\seq{B}^+_f$. It follows that $\seq{A}^+$ is local-global convergent, what concludes our proof.
\end{proof}
\subsection{Local Global Quasi-Limits}
\label{sec:quasi}
Let us finish this paper in an ambitious way. In \cite{limit1,modeling} we defined the notion of modeling as a limit object from structural convergence.

Modeling limits generalize graphing limits and thus it follows from \cite{CMUC} that ${\rm FO}$-convergent sequences of graphs with bounded degrees have modeling limits. In \cite{limit1} we constructed modeling limits for ${\rm FO}$-convergent sequences of graphs with bounded tree-depth, and extended the construction to ${\rm FO}$-convergent sequences of trees in \cite{modeling}. Then existence of modelings for ${\rm FO}$-convergent sequences has been proved for graphs with bounded path-width \cite{gajarsky2016first} and eventually for sequences of graphs in an arbitrary nowhere dense class \cite{modeling_jsl}, which is best possible when considering monotone classes of graphs \cite{modeling}. In fact this provides us with a high level analytic characterization of nowhere dense classes.

\begin{definition}
Let $\seq{A}$ be a local-global convergent sequence.
A modeling $\mathbf{L}$ is a {\em local-global quasi-limit} of $\seq{A}$ if for every local convergent sequence $\seq{A^+}$ of lifts of $\seq{A}$  and every $\epsilon>0$ there exists an admissible lift $\mathbf L^+$ of $\mathbf L$ (that is a lift  $\mathbf L^+$ of $\mathbf L$ that is a modeling),  such that for every local formula $\phi$ we have  
\[	\bigl|\langle\phi,\mathbf L^+\rangle-\lim_{n\rightarrow\infty}\langle\phi,\mathbf A_n^+\rangle\bigr|<c_\phi \epsilon,
\]
where $c_\phi$ is a positive constant, which depends only on $\phi$.
\end{definition}

In other words, the closure of the measures associated to admissible lifts of $\mathbf L$ includes the limit Hausdorff limit of the sets of measures associated to lifts of the sequence.

For local-global convergence, it was proved in \cite{hatami2014limits} that graphings still suffice as limit objects. We don't know, however, if every local-global convergent sequence of graphs in a nowhere dense class has a modeling local-global limit. We close this paper by proving that this is almost the case, in the sense that every local-global convergent sequence of graphs in a nowhere dense class has a modeling local-global quasi-limit.

We consider a fixed countable signature $\sigma$ and the signature $\tau$ obtained by adding countably many unary symbols $M_1,M_2,\dots,M_n,\dots$ to $\sigma$, and the forgetful interpretation $\mathsf{Sh}:\tau\rightarrow\sigma$. As before we understand local-global convergence as $\mathsf{Sh}_*({\rm FO}^{\rm local})$-convergence.
We fix a chain covering $\mathfrak X$ of ${\rm FO}^{\rm local}(\tau)$ (see Section~\ref{sec:rep}), from which we derive metrics and pseudo-metrics as in Sections~\ref{sec:met} and~\ref{sec:LHC}. We also fix a bijection $\beta:\mathbb N\times \mathbb N\rightarrow\mathbb N$ (for Hilbert hotel argument) and let 
$Z_c$ be the renaming interpretation which renames $M_{\beta(c,i)}$ as $M_i$ and forget all the marks not being renamed.

\begin{lemma}
\label{lem:netseq}
There exists a function $h:(0,1)\rightarrow \mathbb N$ with the following property:

For every local-global convergent sequence  $\seq{A}$ of $\sigma$-structures 
there exists a local convergent sequence $\seq{B}$ of $\tau$-structures with $P(\mathsf{Sh})(\seq{B})=\seq{A}$,  such that for every $\epsilon>0$ there exists some integer $n_0$  such that 

for every $n\geq n_0$ and every $\mathbf C\in  P(\mathsf{Sh})^{-1}(\mathbf A_n)$ there exists $1\leq c\leq h(\epsilon)$ with
\[
{\rm dist}_{\mathfrak X}(\mathbf C,P(\mathsf Z_c)(\mathbf B_n))<\epsilon.
\]
\end{lemma}
\begin{proof}
	As the space $({\rm Rel}(\tau),{\rm dist}_{\mathfrak X})$ is totally bounded there exists a mapping $g:(0,1)\rightarrow \mathbb N$ such that for each $\epsilon>0$ and each $\sigma$-structure $\mathbf A$ there is a subset $\mathcal B_{\mathbf A,\epsilon}$ of $P(\mathsf{Sh})^{-1}(\mathbf A)$ of cardinality at most $g(\epsilon)$ with the property that every $\mathbf C\in P(\mathsf{Sh})^{-1}(\mathbf A)$ is at ${\rm dist}_{\mathfrak X}$-distance at most $\epsilon$ from a $\tau$-structure in $\mathcal B_{\mathbf A,\epsilon}$. (Such a set may be called an $\epsilon$-covering.) We construct an infinite sequence $(\mathbf A^{(i)})_{i\in\mathbb N}$ of $\tau$-structures
	by 
	listing all the structures in $\mathcal B_{\mathbf A,1/2}$ then all the structures in $\mathcal B_{\mathbf A,1/4}$, etc.
	
	We now construct a $\tau$-structure $\mathbf A^+\in P(\mathsf{Sh})^{-1}(\mathbf A)$ by letting $M_{\beta(i,j)}(\mathbf A^+)=M_j(\mathbf A^{(i)})$. Hence $\mathbf A^{(i)}=P(\mathsf Z_i)(\mathbf A^+)$. We say that $\mathbf A^+$ is a {\em universal lift} of $\mathbf A$.
	
Define the function $h:(0,1)\rightarrow \mathbb N$ 
by
\[
h(x)=\sum_{i=1}^{\lceil-\log x\rceil+1}g(2^{-i}).
\]
Then for every $\epsilon>0$ and every $\mathbf B\in P(\mathsf{Sh})^{-1}(\mathbf A)$  there is an index $c\leq h(\epsilon)$ such that ${\rm dist}_{\mathfrak X}(\mathbf C,\mathbf A^{(c)})<\epsilon/2$, that is such that 
		${\rm dist}_{\mathfrak X}(\mathbf C,P(\mathsf Z_c)(\mathbf A^+))<\epsilon/2$. 

Now consider the local-global convergent sequence $\seq{A}$ and a sequence $\seq{A}^+$ where $\mathbf A_n^+$ is a universal lift of $\mathbf A_n$. This last sequence has a local convergent subsequence $\seq{A}^+_f$, which we extend into a sequence $\seq{B}$ lifting $\seq{A}$.

Let $\epsilon>0$. According to local-global convergence of $\seq{A}$ and local convergence of $\seq{B}$ there exists $n_0$ such that for every $n,m\geq n_0$ 
we have ${\rm dist}_{\mathfrak X,\mathsf{Sh}}^{\rm H}(\mathbf A_n,\mathbf A_m)<\epsilon/4$ and 
${\rm dist}_{\mathfrak X}(\mathbf B_n,\mathbf B_m)<\alpha$,
where $\alpha$ is such that for every $i\leq h(\epsilon/2)$ we have \[{\rm dist}_{\mathfrak X}(\mathbf X,\mathbf Y)<\alpha\Rightarrow {\rm dist}_{\mathfrak X}(P(\mathsf Z_i)(\mathbf X),P(\mathsf Z_i)(\mathbf Y))<\epsilon/4.\] 
Let $n\geq n_0$ (hence $f(n)\geq n_0$). 
Let $\mathbf C\in P(\mathsf{Sh})^{-1}(\mathbf A_n)$.
Then there exists $\mathbf C'\in P(\mathsf{Sh})^{-1}(\mathbf A_{f(n)})$ such that 
${\rm dist}_{\mathfrak X}(\mathbf C,\mathbf C')<\epsilon/4$. As $\mathbf B_{f(n)}=\mathbf A_{f(n)}^+$ is a universal lift of $\mathbf A_{f(n)}$ there exists $c\leq h(\epsilon)$ such that ${\rm dist}_{\mathfrak X}(P(\mathsf Z_c)(\mathbf B_{f(n)}),\mathbf C')<\epsilon/2$.
As ${\rm dist}_{\mathfrak X}(\mathbf B_n,\mathbf B_{f(n)})<\alpha$ we have
${\rm dist}_{\mathfrak X}(P(\mathsf Z_c)(\mathbf B_{f(n)}),P(\mathsf Z_c)(\mathbf B_{n})<\epsilon/4$. Altogether, we get
${\rm dist}_{\mathfrak X}(P(\mathsf Z_c)(\mathbf B_{n}),\mathbf C)<\epsilon$ as wanted.
\end{proof}

\begin{definition}
A $\sigma$-modeling $\mathbf L$ is a {\em quasi-limits} of a local global convergent sequence $\seq{A}$ of $\sigma$-structures if, for every local convergent sequence $\seq{A}^+$ of $\tau$-structures with $P(\mathsf{Sh})(\seq{A}^+)=\seq{A}$ and for every $\epsilon>0$ there exists a $\tau$-modeling $\mathbf L^+$ with $P(\mathsf{Sh})(\mathbf L^+)=\mathbf L$ such 
that $\limsup{\rm dist}_{\mathfrak X}(\mathbf L^+,\seq{A}^+)<\epsilon$.
	\end{definition}

In other words, for every local-global convergent sequence $\seq{A}$ there is a modeling $\mathbf L$ such that 
any local convergent sequence lifting $\seq{A}$ has a limit which is $\epsilon$-close to an admissible lifting of $\mathbf L$. (By admissible, we mean that the lift of $\mathbf L$ is itself a modeling.)

\begin{theorem}
\label{thm:quasilimit}
	Every local-global convergent sequence of graphs in a nowhere dense class has a modeling quasi-limit.
\end{theorem}
\begin{proof}
Let $\seq{A}$ be a local-global convergent of graphs in a nowhere dense class.
According to Lemma~\ref{lem:netseq}
there exists a local convergent sequence $\seq{B}$ of marked graphs with $P(\mathsf{Sh})(\seq{B})=\seq{A}$,  such that for every $\epsilon>0$ there exists some integer $n_0$  such that 
for every $n\geq n_0$ and every $\mathbf C\in  P(\mathsf{Sh})^{-1}(\mathbf A_n)$ there exists $1\leq c\leq h(\epsilon)$ with
\[
{\rm dist}_{\mathfrak X}(\mathbf C,P(\mathsf Z_c)(\mathbf B_n))<\epsilon.
\]
According to \cite{modeling_jsl} the sequence $\seq{B}$ has a modeling limit
$\mathbf L^+$. Then $\mathbf L=P(\mathsf{Sh}(\mathbf L^+)$ is a modeling quasi-limit of $\seq{A}$. 
\end{proof}

We conjecture that it is possible to refine the  notion of admissible lift and get the reverse direction.

\begin{conjecture}
	Every local-global convergent sequence of graphs in a nowhere dense class has a  modeling limit.
\end{conjecture}
\providecommand{\bysame}{\leavevmode\hbox to3em{\hrulefill}\thinspace}
\providecommand{\MR}{\relax\ifhmode\unskip\space\fi MR }
\providecommand{\MRhref}[2]{%
  \href{http://www.ams.org/mathscinet-getitem?mr=#1}{#2}
}
\providecommand{\href}[2]{#2}

\end{document}